\documentclass[12pt,a4paper]{amsart}

\newtheorem{theorem}{Theorem}[section]
\newtheorem{corollary}[theorem]{Corollary}
\newtheorem{lemma}[theorem]{Lemma}
\newtheorem{proposition}[theorem]{Proposition}
\theoremstyle{definition}
\newtheorem{definition}[theorem]{Definition}
\newtheorem{example}[theorem]{Example}
\newtheorem{remark}[theorem]{Remark}

\usepackage{amsfonts}
\usepackage[figuresright]{rotating}
\usepackage{amssymb}
\usepackage{amsmath}
\usepackage{amsmath}
\usepackage{graphics}

\title{Probability distributions with rational free $R$-transform}
\date{\today}

\author{Wojciech M{\l}otkowski}
\address{Instytut Matematyczny,
Uniwersytet Wroc{\l}awski,
Plac~Grunwaldzki~2/4,
50-384 Wroc{\l}aw, Poland}
\email{mlotkow@math.uni.wroc.pl}

\begin{document}

\begin{abstract}
We study the class $\mathcal{M}_{\mathrm{ratio}}$ of those probability distributions for which the free $R$-transforms are rational functions. This class is closed under the additive free convolution, additive free powers and under the monotone convolution. We prove a sufficient condition that a rational function is the free $R$-transform of a probability distribution. Several examples are provided, including that of free deconvolution.
\end{abstract}

\maketitle

\tableofcontents
\newpage
\section*{Introduction}

The aim of this paper is to study the class $\mathcal{M}_{\mathrm{ratio}}$ of such probability distributions $\mu$ for which the free $R$-transform $R_{\mu}(w)$ is a rational function.
In another words, for a given rational function $F(w)$ we ask whether there exists a probability
distribution $\mu$ on $\mathbb{R}$ for which the moment generating function
\[M_{\mu}(z):=\int_{\mathbb{R}}\frac{\mu(dx)}{1-xz}\] satisfies equation $F(z M_{\mu}(z))=z$ in a neighbourhood of $z=0$.
If this is the case then $\mu$ has compact support and $R_{\mu}(w)=w/F(w)-1$.

Let $\mathcal{F}$ denote the class of rational functions of the form $F(w)=wP(w)/Q(w)$ for some relatively prime polynomials $P(w),Q(w)\in\mathbb{R}[w]$, satisfying $P(0)=Q(0)=1$.
We are going to study three subsets of $\mathcal{F}$: the set $\mathcal{F}_{\mathrm{dist}}$ of these $F$ that there exists $\mu\in\mathcal{M}_{\mathrm{ratio}}$ such that $F(z M_{\mu}(z))=z$ holds in a neighbourhood of $z=0$ (then we write $F=F_{\mu}$ and $\mu=\mu(F)$), the set $\mathcal{F}_{\mathrm{rr}}$ of these $F$ that a certain set $\mathcal{N}(F)$ (see Definition~\ref{def:N}) is contained in $\mathbb{R}$, and the set $\mathcal{F}_{\mathrm{rr}}^{0}$
of these $F$ that the numerator of $F'(w)$ (called the \textit{characteristic polynomial of $F$}, and denoted $\chi_{F}(w)$) has only real roots.
We will prove that
\[
\mathcal{F}_{\mathrm{rr}}^{0}\varsubsetneq
\mathcal{F}_{\mathrm{rr}}\varsubsetneq
\mathcal{F}_{\mathrm{dist}}\varsubsetneq
\mathcal{F},
\]
see Theorem~\ref{thm:Ndist}, Proposition~\ref{pro:rr0}, Example~\ref{ex:rrminus} and Example~\ref{ex:darkmatter}.
A special case, when $Q(w)=1$, was studied in \cite{LM2020}.
The families of distributions corresponding to $\mathcal{F}_{\mathrm{rr}}^{0}$ and $\mathcal{F}_{\mathrm{rr}}$ will be denoted $\mathcal{M}_{\mathrm{rr}}^{0}$ and $\mathcal{M}_{\mathrm{rr}}$ respectively.

For $F\in\mathcal{F}$ we define its $R$-transform $R(w):=w/F(w)-1$.
If $F=F_{\mu}$ for some $\mu\in\mathcal{M}_{\mathrm{ratio}}$ then $R$ coincides with the free $R$-transform of $\mu$.
For $F,F_1,F_2\in\mathcal{F}$, with the $R$-transforms $R,R_1,R_2$, and for
$t\in\mathbb{R}$, we define $F^{\boxplus t},F_1\boxplus F_2\in\mathcal{F}$ as these elements of $\mathcal{F}$ whose $R$-transforms are $t\cdot R(w)$ and $R_1(w)+R_2(w)$.
If $F=F_{\mu}$, $F_1=F_{\mu_1}$, $F_2=F_{\mu_2}$ for some $\mu,\mu_1,\mu_2\in\mathcal{M}_{\mathrm{ratio}}$ then $F^{\boxplus t}=F_{\mu^{\boxplus t}}$ (provided $\mu^{\boxplus t}$ exists) and $F_1\boxplus F_2=F_{\mu_1\boxplus\mu_2}$. Therefore $\mathcal{M}_{\mathrm{ratio}}$ is closed under the free additive convolution and under the free additive powers $\mu^{\boxplus t}$ with $t\ge1$.
Moreover, it is closed under monotone convolution, namely $F_{\mu_1\rhd\mu_2}=F_{\mu_2}\circ F_{\mu_1}$.

For fixed $F\in\mathcal{F}$ it is interesting to study the free powers $F^{\boxplus t}$, $t>0$. It may happen, that $F^{\boxplus t_0}\in\mathcal{F}_{\mathrm{rr}}^{0}$ for some $t_0>0$ and there exists $\delta>0$ such that $F^{\boxplus t}\notin\mathcal{F}_{\mathrm{rr}}^{0}$ either for all $t\in(t_0-\delta,t_0)$, or for all $t\in(t_0,t_0+\delta)$.
Then we say that $F$ admits \textit{phase transition at $t_0$.} If $\chi_t(w)$ denotes the characteristic polynomial of $F^{\boxplus t}$
and $F$ admits phase transition at $t_0$ then either $\deg \chi_{t_0}(w)\le\deg\chi_{t}(w)-2$ for $t\ne t_0$ or
$\chi_{t_0}(w)$ has a multiple root. In the latter case we will say that $F\in\mathcal{F}$ is \textit{singular}.

The methods developed here allow us in particular:
\begin{itemize}
    \item to provide two families of free deconvolution: one involving Wigner law and Mar\-chen\-ko-Pastur law, and one involving two Marcheno-Pastur laws (Section~\ref{sec:deconv}, see also Section~\ref{sec:monotone}),
    \item to prove, that for every positive integer $k$ there exists a constant $C>0$ such that $r_n=C\cdot n^k$ is the free cumulant sequence for some probability distribution (Section~\ref{sec:euler}),
    \item to prove positive definiteness for some integer sequences which are encountered in enumerative combinatorics (graphs, pattern avoiding permutations, lattice paths, rooted plane trees, noncrossing diagrams, operads, etc.).
\end{itemize}

The paper is organized as follows. First, we describe the family $\mathcal{F}$ together with some operations on it: the translations,
the dilations, the free product, the free powers, and the composition.
In Section~\ref{sec:posdef} we prove a sufficient condition for positive definiteness of the sequence $s_n(F)$ corresponding to $F\in\mathcal{F}$ (Theorem~\ref{thm:Ndist}) and define subfamilies $\mathcal{F}_{\mathrm{rr}}^{0},\mathcal{F}_{\mathrm{rr}},\mathcal{F}_{\mathrm{dist}}$ of $\mathcal{F}$.
Then we describe relations with the free and monotone convolution of probability distributions on $\mathbb{R}$.
In Section~\ref{sec:inf} we describe these freely infinitely divisible distributions which belong to $\mathcal{M}_{\mathrm{ratio}}$ (Theorem~\ref{thm:levy}). We prove that for $F\in\mathcal{F}$ we have $F=F_{\mu}$ for some freely infinitely divisible distribution $\mu$ if and only if $F^{\boxplus t}\in\mathcal{F}_{\mathrm{rr}}^{0}$ for sufficiently small $t>0$ (Theorem~\ref{thm:infdivstars}).

In Section~\ref{sec:deconv} we study two families of free deconvolution, i.e. distributions $\mu$ such that $R_{\mu}(w)=R_{\mu_1}(w)-R_{\mu_2}(w)$ for some distributions $\mu_1,\mu_2$. The first family involves one Wigner and one Marchenko-Pastur distribution as $\mu_1,\mu_2$, the second family involves two Marchenko-Pastur distributions. In both cases the corresponding $F\in\mathcal{F}_{\mathrm{rr}}^{0}$ is singular.
Further examples are obtained in Section~\ref{sec:monotone} as results of monotone convolution.

Next section is devoted to polynomials $R(w)$ which are free $R$-transforms of some distributions.
Such polynomials were studied  by Chistyakov and G\"{o}tze \cite{chistyakov2011}.
We will see that if $\deg R\le3$, $R(0)=0$, then $R$ is the free $R$ transform of a distribution if and only if the corresponding $F\in\mathcal{F}$ belongs to $F\in\mathcal{F}_{\mathrm{rr}}^{0}$, however if $\deg R=4$ then the corresponding $F$ can belong to $\mathcal{F}_{\mathrm{dist}}\setminus\mathcal{F}_{\mathrm{rr}}$, see Figure~\ref{fig:galaxy}.

In Section~\ref{sec:euler} we prove that for every integer $k\ge1$ there exists a constant $C$ such that $r_n=C\cdot n^k$ is the free cumulant sequence of a certain probability distribution. The proof relies on a result of Conger~\cite{conger2010,conger2007thesis},
which says that some version of the Eulerian polynomials have only real simple roots. In particular, for $k=1$ one can take $C=27/8$, see Example~\ref{ex:rnn1}, for the case $k=2$ see Example~\ref{ex:euler2}.

In Section~\ref{sec:polynomials} we investigate these $F\in\mathcal{F}$ which are polynomials, and free powers $F^{\boxplus t}$ of such elements~$F$.
Finally we provide a list of examples of integer sequences of the form $s_n(F)$, with $F\in\mathcal{F}_{\mathrm{rr}}^{0}$,
which can be found in the On-line Encyclopedia of Integer Sequences~\cite{oeis}~(OEIS).

\section{The class $\mathcal{F}$}\label{sec:classf}

Let $\mathcal{F}$ denote the class of all rational functions
of the form
\[
F(w)=\frac{wP(w)}{Q(w)},
\]
where $P(w),Q(w)$ are relatively prime polynomials, with real coefficients, such that $P(0)=Q(0)=1$.
Then $F(0)=0$, $F'(0)=1$.
We define the \textit{$R$-transform} of $F$ by
\begin{equation}\label{for:rtransformf}
R(w):=\frac{w}{F(w)}-1=\frac{Q(w)-P(w)}{P(w)},
\end{equation}
so that
\begin{equation}\label{for:frtransform}
F(w)=\frac{w}{1+R(w)}.
\end{equation}

For $u,c\in\mathbb{R}$, $c\ne0$, we define \textit{translation of $F$ by $u$}:
\begin{equation}\label{for:translation}
F_1(w):=\frac{F(w)}{1+uF(w)}=\frac{w P(w)}{Q(w)+uwP(w)},
\end{equation}
and \textit{dilation of $F$ by $c$}:
\begin{equation}\label{for:dilation}
F_2(w):=\frac{F(cw)}{c}=\frac{wP(cw)}{Q(cw)}.
\end{equation}
It is easy to check that the $R$-transforms corresponding to $F_1,F_2$ are
\begin{equation}\label{eq:rtransdil}
R_1(w)=R(w)+uw\quad \hbox{and}\quad R_2(w)=R(cw).
\end{equation}

Now we introduce some further operations on the class $\mathcal{F}$.

\begin{definition}
Assume that $F,F_1,F_2\in\mathcal{F}$:
\[
F(w)=\frac{wP(w)}{Q(w)},\quad
F_{1}(w)=\frac{wP_1(w)}{Q_1(w)},\quad
F_{2}(w)=\frac{wP_2(w)}{Q_2(w)},
\]
and $t\in\mathbb{R}$. Then we define
\begin{align}
\left(F_{1}\boxplus F_{2}\right)(w)&:=\frac{wF_1(w)F_2(w)}{wF_1(w)+wF_2(w)-F_1(w)F_2(w)}\\
&\phantom{:}=\frac{wP_1(w) P_2(w)}{P_1(w) Q_2(w)+P_2(w) Q_1(w)-P_1(w)P_2(w)},\\
\left(F^{\boxplus t}\right)(w)&:=\frac{wF(w)}{tw+(1-t)F(w)}=
\frac{w P(w)}{P(w)+t\big(Q(w)-P(w)\big)},\label{eq:freepower}\\
\left(F_2\circ F_1\right)(w)&:=F_2(F_1(w)),
\end{align}
called, respectively, the \textit{free convolution} of $F_1$ and $F_2$,
the \textit{free power} of $F$, and the \textit{composition} of $F_2$ and $F_1$.
\end{definition}

It is easy to check the following property:

\begin{proposition}
If $R,R_1,R_2$ are the $R$-transforms of $F,F_1,F_2$, $t\in\mathbb{R}$ then the
$R$-transforms of $F^{\boxplus t}$ and of $F_1\boxplus F_2$
are $t\cdot R(w)$ and $R_1(w)+R_2(w)$ respectively.
\end{proposition}

Note that $\left(\mathcal{F},\boxplus,w\right)$ is an abelian group and $(\mathcal{F},\circ,w)$ is a unital semigroup,
the only invertible elements are of the form $F(w)=w/(1+uw)$, $u\in\mathbb{R}$.

\section{Positive definite sequences}\label{sec:posdef}

If $F\in\mathcal{F}$ then $F$ admits the composition inverse $D(z)$, so that $F(D(z))=z$ in a neighborhood of $z=0$, and $D(z)$ can be represented as the power series
\[
D(z)=z\sum_{n=0}^{\infty}s_n z^{n},
\]
with $s_0=1$. We will call the coefficients $s_n$ the \textit{moments of $F$} and denote $s_n:=s_n(F)$.
One can check that if
\[
P(w)=1+a_1 w+\ldots+a_p w^p,\qquad
Q(w)=1+b_1 w+\ldots+b_q w^q
\]
then
\begin{equation}\label{eq:firstmoments}
s_0=1,\quad
s_1=b_1-a_1,\quad
s_2=2a_1^2+b_1^2-3a_1 b_1-a_2+b_2.
\end{equation}

It is of interest to know whether the sequence $s_n$ is positive definite,
equivalently, if there exists a probability distribution $\mu$ on $\mathbb{R}$
such that $s_n$ is the moment sequence of $\mu$,~i.e.
\[
s_n=\int_{\mathbb{R}} t^n\,\mu(dt),
\]
$n=0,1,\ldots$, so that
\begin{equation}\label{eq:dintegral}
D(z):=\int_{\mathbb{R}}\frac{z\mu(dt)}{1-tz}.
\end{equation}
In such a case we set $F_{\mu}(w):=F(w)$, $R_{\mu}(w):=R(w)$, $D_{\mu}(z):=D(z)$, $s_n(\mu):=s_n(F)$ and $\mu(F):=\mu$.
If $\mu_1$ (resp. $\mu_2$) is the translation of $\mu$ by $u\in\mathbb{R}$ (resp. the dilation of $\mu$ by $c\ne0$) then the corresponding $F_1$ (resp. $F_2$) in $\mathcal{F}$ is given by (\ref{for:translation}) (resp. by (\ref{for:dilation})).
The variance $s_2-s_1^2$ of $\mu$ is nonnegative and can be $0$ only if $\mu=\delta_u$ for some $u\in\mathbb{R}$, which, by (\ref{eq:firstmoments}), leads to the following necessary condition of positive definiteness:
\begin{equation}\label{for:necessary}
a_1^2-a_1 b_1-a_2+b_2\ge0.
\end{equation}
Since $D(z)$ is analytic in a neighborhood of $z=0$, $\mu(F)$ has compact support and hence is uniquely determined by $F$.
For example, if $F(w)=w/(1+uw)$, with $u\in\mathbb{R}$, then $s_n=u^n$ and the corresponding measure $\mu$ is $\delta_{u}$.
The class of all such $F\in\mathcal{F}$ that $s_n(F)$ is positive definite we will denote by $\mathcal{F}_{\mathrm{dist}}$.

\begin{proposition}
Assume that $F\in\mathcal{F}_{\mathrm{dist}}$ and that $\mu(F)$ is not of the form $\delta_u$, $u\in\mathbb{R}$. If $t<0$ then $F^{\boxplus t}\notin\mathcal{F}_{\mathrm{dist}}$.
\end{proposition}

\begin{proof}
By assumption the variance $a_1^2-a_1 b_1-a_2+b_2$ of $\mu(F)$ is positive and using (\ref{eq:freepower})
it is easy to check that if $\mu(F^{\boxplus t})$ exits then its variance is
$t(a_1^2-a_1 b_1+b_2-a_2)$, hence $t>0$.
\end{proof}

\begin{definition}\label{def:N}
By $\mathcal{N}(F)$ we will denote the set of all $z_0\in\mathbb{C}$ such that the polynomial
\[
wP(w)-z_0 Q(w)
\]
admits a multiple root $w_0$. Equivalently, $F(w_0)=z_0$ and $F'(w_0)=0$.
\end{definition}

We can now formulate our main result, which is a generalisation of Theorem~2.6 in \cite{LM2020}.

\begin{theorem}\label{thm:Ndist}
If $\mathcal{N}(F)\subseteq\mathbb{R}$ then $F\in\mathcal{F}_{\mathrm{dist}}$.
\end{theorem}

\begin{proof}
Since $w=0$ is a single root of the equation $F(w)=0$, and $F'(0)=1$, $F$ has unique composition inverse $D(z)$ on a neighbourhood $U$ of $z=0$, with $D(0)=0$. Moreover, $D(z)$ can be extended to $U\cup(\mathbb{C}\setminus\mathbb{R})$, because of the assumption that $\mathcal{N}(F)\subseteq\mathbb{R}$. In addition, $D'(0)=1$, which implies that if $z\in\mathbb{C}_{+}:=\{z\in\mathbb{C}:\mathrm{Im}z>0\}$ (resp. $z\in\mathbb{C}_{-}:=\{z\in\mathbb{C}:\mathrm{Im}z<0\}$)
and $|z|$ is sufficiently small then $D(z)\in\mathbb{C}_{+}$ (resp. $D(z)\in\mathbb{C}_{-}$). Since $F(D(z))=z$, if $D(z)$ is real then so is $z$, in another words, $\mathrm{Im} D(z)$ never vanishes on $\mathbb{C}\setminus\mathbb{R}$. Consequently, $D(z)$ maps $\mathbb{C}_{+}$ into $\mathbb{C}_{+}$ and $\mathbb{C}_{-}$ into $\mathbb{C}_{+}$, i.e. $D(z)$ is a Pick function, $D(1/z)$ is the Cauchy transform of a probability distribution $\mu$ on $\mathbb{R}$,
and therefore (\ref{eq:dintegral}) holds.
\end{proof}

It is easy to see that Definition~\ref{def:N} can be formulated in the following way:

\begin{proposition}
A complex number $z_0$ belongs to $\mathcal{N}(F)$ if and only if there is $w_0\in\mathbb{C}$
such that
\begin{equation}\label{eq:charsystem}
\begin{split}
w_0 P(w_0)&=z_0 Q(w_0),\\
P(w_0)+w_0 P'(w_0)&=z_0 Q'(w_0).
\end{split}
\end{equation}
\end{proposition}

We will call (\ref{eq:charsystem}) the \textit{characteristic system of equations of $F$.}

\begin{example}[Catalan numbers $A000108$]
The generating function for the Catalan numbers satisfies: $M(z)=1+z M(z)^2$, which for $D(z):=zM(z)$ becomes $z=D(z)-D(z)^2$.
Take $F(w)=w-w^2$, so that $P(w)=1-w$, $Q(w)=1$. Then the characteristic system  (\ref{eq:charsystem}) is: $w(1-w)=z$ and $1-2w=0$, and the only solution is $(w,z)=(1/2,1/4)$. This is yet another proof that the Catalan sequence $s_n(F)$, see entry $A000108$ in OEIS \cite{oeis}, is positive definite.
\end{example}

Recall that we assume that $P(w)$, $Q(w)$ are relatively prime.

\begin{proposition}
If $(w_0,z_0)$ is a solution of (\ref{eq:charsystem}) then $w_0\ne0$ and $Q(w_0)\ne0$. If, in addition, $z_0\ne0$ then also $P(w_0)\ne0$. On the other hand, if $z_0=0$ then $P(w_0)=P'(w_0)=0$.

If $(w_0,z_0)$ is a solution of (\ref{eq:charsystem}) then $(\overline{w}_0,\overline{z}_0)$ is also a solution and if $w_0$ is real then so is $z_0$.

Assume that $F_1$, $F_2$ is the translation of $F$ by $u\ne0$ and the dilation of $F$ by $c\ne0$, respectively, and that $(w_0,z_0)$ is a solution of the characteristic system (\ref{eq:charsystem}) of $F$. Then $(w_0/c,z_0/c)$ is a solution of the characteristic system of $F_2$.
If $u z_0\ne-1$ then $\big(w_0,z_0/(1+uz_0)\big)$ is a solution of the characteristic system for $F_1$. If $uz_0=-1$ then there is no $z_1\in\mathbb{C}$ such that $(w_0,z_1)$ is a solution of the system of $F_1$.
\end{proposition}

\begin{proof}
Assume that $(w_0,z_0)$ is a solution of (\ref{eq:charsystem}). If $w_0=0$ then, from the first equation, $z_0=0$ as $P(0)=Q(0)=1$, but this leads to a contradiction in the second equation.
If $z_0\ne0$ then $P(w_0)=0$ (resp. $Q(w)=0$) would imply $Q(w_0)=0$ (resp. $P(w)=0$) from the first equation.
\end{proof}

In most cases studied in this paper we will apply the following criterion.

\begin{proposition}\label{pro:rr0}
If $(w_0,z_0)$ is a solution of (\ref{eq:charsystem}) then $w_0$ is a root of the polynomial
\begin{equation}\label{for:chacharpolynomial}
\chi_F(w):=
\big(P(w)+wP'(w)\big)Q(w)-w P(w) Q'(w).
\end{equation}
Consequently, if $\chi_{F}(w)$ has only real roots then the sequence $s_n(F)$ is positive definite.
\end{proposition}

\begin{proof}
If $z_0=0$ then (\ref{eq:charsystem}) implies that $P(w_0)=P'(w_0)=0$, as $w_0\ne0$, and then $\chi_{F}(w_0)=0$. For $z_0\ne0$ the statement is obvious.
\end{proof}

We will call $\chi_{F}(w)$ the \textit{characteristic polynomial of $F$}.

\begin{example}[$A048779$]\label{ex:A048779}
Take $F(w):=w(1-w)(1-2w+2w^2)$. Then $\chi_{F}(w)=(1-2w)^3$
and hence the sequence $s_n(F)=A048779$ is positive definite.
\end{example}

\begin{example}
Take $P(w)=1+a_1 w$, $Q(w)=1+b_1 w+b_2 w^2$. Then
\[\chi_{F}(w)=1+2a_1 w+a_1 b_1 w^2-b_2 w^2.\]
Combining Proposition~\ref{pro:rr0} and the necessary condition (\ref{for:necessary}), we see that:
$\chi_{F}(w)$ has only real roots if and only $F\in\mathcal{F}_{\mathrm{dist}}$, equivalently: $a_1^2-a_1 b_1+b_2\ge0$.
\end{example}

\begin{remark}\label{remark:qmultiple}
Note, that if $w_0$ is a multiple root of $Q(w)$, i.e. $Q(w_0)=Q'(w_0)=0$, then $w_0$ is a root of $\chi_{F}(w)$, but there is no $z_0$ such that $(w_0,z_0)$ is a solution of (\ref{eq:charsystem}).
\end{remark}

If $\mu(F)$ exists then we set $\chi_{\mu(F)}(w):=\chi_{F}(w)$.
Denote by $\mathcal{F}_{\mathrm{dist}}$ the class of all $F\in\mathcal{F}$
such that $s_n(F)$ is positive definite and by $\mathcal{M}_{\mathrm{ratio}}$
the class of the corresponding distributions $\mu(F)$.
Moreover, denote by $\mathcal{F}_{\mathrm{rr}}$ (respectively, $\mathcal{F}_{\mathrm{rr}}^{0}$) the set of all $F\in\mathcal{F}$ such that $\mathcal{N}(F)\subseteq\mathbb{R}$ (respectively, that all the roots of $\chi_F(w)$
are real), and the corresponding families of distributions
by $\mathcal{M}_{\mathrm{rr}}$ and $\mathcal{M}_{\mathrm{rr}}^{0}$, respectively.
We will see that
\[
\mathcal{F}_{\mathrm{rr}}^{0}\varsubsetneq
\mathcal{F}_{\mathrm{rr}}\varsubsetneq
\mathcal{F}_{\mathrm{dist}}\varsubsetneq
\mathcal{F},
\]
consequently  $\mathcal{M}_{\mathrm{rr}}^0\varsubsetneq
\mathcal{M}_{\mathrm{rr}}\varsubsetneq
\mathcal{M}_{\mathrm{ratio}}$.

\begin{example}[$A121988$]\label{ex:rrminus}
Take $F(w)=w(1-w)(1-w+w^2)$. Then the solutions $(w,z)$ of (\ref{eq:charsystem}) are
\[
\left(\frac{1}{2},\frac{3}{16}\right),\quad
\left(\frac{1}{2}\pm\frac{\mathrm{i}}{2},\frac{1}{4}\right)
\]
and the characteristic polynomial is
\[
\chi_{F}(w)=(1-2w)(1-2w+2w^2).
\]
Therefore the sequence $s_n(F)$ is positive definite and $F\in\mathcal{F}_{\mathrm{rr}}\setminus\mathcal{F}_{\mathrm{rr}}^{0}$.
\end{example}

Now we observe that the class $\mathcal{F}_{\mathrm{rr}}$ is closed under composition, i.e. $\left(\mathcal{F}_\mathrm{rr},\circ,w\right)$ is a unital semigroup.

\begin{proposition}\label{pro:monotonerr}
If $F_1,F_2\in\mathcal{F}_{\mathrm{rr}}$ then $F_2\circ F_1\in\mathcal{F}_{\mathrm{rr}}$.
\end{proposition}

\begin{proof}
If $z_0\in\mathcal{N}(F_2\circ F_1)$ then $F_2(F_1(w_0))=z_0$ for some $w_0\in\mathbb{C}$
such that $F'_2(F_1(w_0))\cdot F'_1(w_0)=0$.
If $F'_2(F_1(w_0))=0$ then $z_0\in\mathcal{N}(F_2)$.
If $F'_1(w_0)=0$ then $F_1(w_0)\in\mathcal{N}(F_1)$, which implies that $F_1(w_0)$ and $F_2(F_1(w_0))$ are real.
Hence $\mathcal{N}(F_2\circ F_1)=\mathcal{N}(F_2)\cup F_2[\mathcal{N}_1']$
where $\mathcal{N}_1'$ denotes the set of $z_0\in\mathcal{N}(F_1)$
such that $Q_2(z_0)\ne0$.
\end{proof}

It is easy to check, using (\ref{for:translation}), that translation doesn't change the characteristic polynomial.

\begin{proposition}\label{prop:charpoltranslation}
If $F_1(w)$ is the translation of $F$ by $u\in\mathbb{R}$
then $\chi_{F_1}(w)=\chi_{F}(w)$.

If $F_2(w)$ is the dilation of $F$ by $c\ne0$
then $\chi_{F_2}(w)=\chi_{F}(cw)$.
\end{proposition}

By abuse of notation, we will sometimes write $\chi_t$ for the characteristic polynomial of $F^{\boxplus t}$ if $F\in\mathcal{F}$ is fixed.

\begin{proposition}
If $F(w)=wP(w)/Q(w)\in\mathcal{F}$, $t\in\mathbb{R}$, then the characteristic polynomial $\chi_t(w)$ of $F^{\boxplus t}$ is
\begin{multline}\label{for:chart}
\chi_{t}(w)=
P(w)^2\\
+t\left[\big(P(w)+wP'(w)\big)\big(Q(w)-P(w)\big)-wP(w)\big(Q'(w)-P'(w)\big)\right].
\end{multline}
\end{proposition}

\subsection{Real-rooted polynomials}\label{subsec:poly}
It will be convenient to understand the topology of real-rooted polynomials.

\newpage
\begin{proposition}
Let $\mathcal{D}_n$ denote the set of all such sequences $(c_1,\ldots,c_n)\in\mathbb{R}^n$
that the polynomial
\[
C(w)=1+c_1 w+\ldots+c_n w^n
\]
has only real roots. Then $\mathcal{D}_n$ is a closed subset of $\mathbb{R}^n$.
Moreover $(c_1,\ldots,c_n)$ belongs to the interior of $\mathcal{D}_n$ if and only if
either
\begin{itemize}
    \item $c_n\ne0$ and all the $n$ roots of $C(w)$ are real and simple, or
    \item $c_n=0$, $c_{n-1}\ne0$ and all the $n-1$ roots of $C(w)$ are real and simple.
\end{itemize}

Consequently, a sequence $(c_1,\ldots,c_n)\in\mathcal{D}_n$ belongs to the boundary of $\mathcal{D}_n$
if and only if either $C(w)$ has a multiple root or $c_n=c_{n-1}=0$.
\end{proposition}

\begin{proof}
For $(c_1,\ldots,c_n)\in\mathcal{D}_n$ we can write
\[
C(w)=1+c_1 w+\ldots+c_n w^n=\prod_{i=1}^{n}(1+v_i w)
\]
for some $v_i\in\mathbb{R}$. The coefficients $c_k$ can be expressed as the elementary symmetric polynomials of $v_1,\ldots,v_n$:
\[
c_k(v_1,\ldots,v_n)=\sum_{1\le i_1<i_2<\ldots<i_k\le n} v_{i_1}v_{i_2}\ldots v_{i_k}
\]
and we have
\[
\det\left(\frac{\partial c_{k}}{\partial v_i}\right)_{i,k=1}^{n}
=\prod_{1\le i<j\le n}(v_i-v_j).\]
Therefore the function
$F:\mathbb{R}^{n}\to\mathcal{D}_{n}$ defined by
\[
(v_1,\ldots,v_n)\mapsto
\big(c_1(v_1,\ldots,v_n),\ldots,c_{n}(v_1,\ldots,v_n)\big)
\]
is invertible in a neighbourhood of any $(v_1,\ldots,v_n)$
satisfying $v_1<v_2<\ldots<v_n$. If one of $v_i$ is $0$ then
the corresponding polynomial $C(w)$ has degree $n-1$ and has $n-1$ real and simple roots.

On the other hand, if $C(w)=\prod_{i=1}^{n}(1+v_i w)$, with $v_i\in\mathbb{R}$, $v_1=v_2$,
then the coefficient sequence of the polynomial
\[
\left(1+2v_1 w+(v_1^2+t)w^2\right)\prod_{i=3}^{n}(1+v_i w),
\]
with $t>0$, can be arbitrarily close to that of $C(w)$.
\end{proof}

Now define $\mathcal{F}(p,q)$ to be the family of all $F(w)=wP(w)/Q(w)\in\mathcal{F}$ such that $\deg P(w)\le p$, $\deg Q(w)\le q$, with the natural topology inherited from $\mathbb{R}^{p+q}$.
Then $\mathcal{F}_{\mathrm{rr}}^{0}\cap\mathcal{F}(p,q)$,
$\mathcal{F}_{\mathrm{rr}}\cap\mathcal{F}(p,q)$ and
$\mathcal{F}_{\mathrm{dist}}\cap\mathcal{F}(p,q)$ are closed subsets of $\mathcal{F}(p,q)$.

\begin{corollary}\label{cor:topology}
If $F\in\mathcal{F}(p,q)$, $\deg \chi_{F}(w)\ge p+q-1$ and all the roots of $\chi_{F}(w)$ are real and simple then $F$ belongs to the interior of $\mathcal{F}_{\mathrm{rr}}^{0}\cap\mathcal{F}(p,q)$
in $\mathcal{F}(p,q)$.

On the other hand, if $F\in\mathcal{F}(p,q)$ belongs to the boundary of $\mathcal{F}_{\mathrm{rr}}^{0}\cap\mathcal{F}(p,q)$ in $\mathcal{F}(p,q)$ then
either $\deg\chi_{F}(w)\le p+q-2$ or $\chi_{F}(w)$ has a multiple root.
\end{corollary}

\begin{definition}
We call an element $F\in\mathcal{F}$ \textit{singular}
if $\chi_{F}$ has a multiple real root.
\end{definition}

\subsection{Phase transition}

\begin{definition}
We say that $F\in\mathcal{F}$ admits \textit{phase transition at $t_0>0$} if $F^{\boxplus t_0}\in\mathcal{F}_{\mathrm{rr}}^{0}$ and there is $\delta>0$ such that either $F^{\boxplus t}\notin\mathcal{F}_{\mathrm{rr}}^{0}$ for $t_0-\delta<t<t_0$ or $F^{\boxplus t}\notin\mathcal{F}_{\mathrm{rr}}^{0}$ for $t_0<t<t_0$.
\end{definition}

Assume that $F\in\mathcal{F}$, $\chi_t(w)$ is the characteristic polynomial of $F^{\boxplus t}$, $\chi_{t_0}(w)$ has all real roots and that $\chi_t(w)$ has some complex, not real roots for $t_0-\delta<t<t_0$ (or for $t_0<t<t_0+\delta$). Writing
\[
\chi_{t}(w)=\prod_{i=1}^{r}\big(1-v_i(t)w\big),
\]
we can assume that each $v_i(t)$ depends continuously on $t\in(t_0-\delta, t_0)$ (or on $t\in(t_0,t_0+\delta)$). Moreover, there are $1\le i<j\le r$ such that $v_i(t)=\overline{v_j(t)}\notin \mathbb{R}$. Then $\lim_{t\to t_0^{-}}v_i(t)=\lim_{t\to t_0^{-}}v_j(t)=v_0$ (or, respectively, $\lim_{t\to t_0^{+}}v_i(t)=\lim_{t\to t_0^{+}}v_j(t)=v_0$), $v_0\in\mathbb{R}$, and then either $v_0\ne0$ and $1/v_0$ is a multiple root of $\chi_{t_0}(w)$ or $v_0=0$ and then the $\deg \chi_{t_0}(w)\le\deg \chi_{t}(w)-2$ for $t\ne t_0$. This justifies the following definition and corollary.

\begin{definition}
We call $t_0>0$ \textit{critical} for $F$ if either $\chi_{t_0}(w)$ has a multiple real root or $\deg \chi_{t_0}(w)\le\deg\chi_{t}(w)-2$ for $t\ne t_0$.
\end{definition}

\begin{corollary}
If $F$ admits phase transition at $t_0>0$ then $t_0$ is critical for $F$.
\end{corollary}

\begin{example}[$A001764$]\label{ex:darkmatter}
Take $P(w)=1-w^2$, $Q(w)=1$. Then
\[
R(w)=\frac{w}{2(1-w)}+\frac{w}{2(1+w)}
\]
so we will see that  $\mu(F)=\mathrm{MP}(1,1/2)\boxplus \mathrm{MP}(-1,1/2)$
in notation of Section~\ref{sec:inf}.
The corresponding moment sequence $s_n(F)$ is aerated $A001764$.

We have
\[
F^{\boxplus t}(w)=\frac{w-w^3}{1+(t-1)w^2}
\]
and the characteristic polynomial of $F^{\boxplus t}$ is
\[
\chi_t(w)=1-(2+t)w^2+(1-t)w^4.
\]
Hence $F^{\boxplus t}\in\mathcal{F}_{\mathrm{rr}}^{0}$ if and only if $0<t\le1$. In fact, $F^{\boxplus t}\notin\mathcal{F}_{\mathrm{rr}}$ for $t>1$
For $t=2$ the solutions of (\ref{eq:charsystem}) are
\begin{align*}
(w,z)&=\left(\pm\sqrt{\sqrt{5}-2},\frac{\pm 1}{2}\left(\sqrt{5}-1\right)\sqrt{\sqrt{5}-2}\right),\\
(w,z)&=\left(\pm\mathbf{i}\sqrt{\sqrt{5}+2},\frac{\pm \mathbf{i}}{2}\left(\sqrt{5}+1\right)\sqrt{\sqrt{5}+2}\right)
\end{align*}
and $s_n(F^{\boxplus2})$ is the aerated sequence $A027307$.
Similarly, $s_n(F^{\boxplus3})$ is aerated $A219535$.
In view of Corollary~\ref{cor:freemonotoneconv}, $F^{\boxplus t}\in\mathcal{F}_{\mathrm{dist}}$,
so that the sequence $s_n\left(F^{\boxplus t}\right)$ is positive definite, for all $t>0$.
Therefore $\mathcal{F}_{\mathrm{rr}}\varsubsetneq
\mathcal{F}_{\mathrm{dist}}$.
\end{example}

\section{Free and monotone convolution}\label{sec:freemonotone}

Let $\mathcal{M}_{\mathrm{c}}$ denote the class of all probability distributions on $\mathbb{R}$ with compact support. For $\mu\in\mathcal{M}_{\mathrm{c}}$ define its \textit{moments}:
\[
s_n(\mu):=\int_{\mathbb{R}}x^n\,\mu(dx),
\]
the \textit{moment generating function}:
\[
M_{\mu}(z):=\sum_{n=0}^{\infty} s_n(\mu) z^n=\int_{\mathbb{R}}\frac{1}{1-xz}\,\mu(dx),
\]
the \textit{upshifted moment generating function} (see \cite{liupego2016}):
\[
D_{\mu}(z):=\sum_{n=0}^{\infty} s_n(\mu) z^{n+1}=\int_{\mathbb{R}}\frac{z}{1-xz}\,\mu(dx),
\]
and the \textit{Cauchy transform}:
\[
G_{\mu}(z):=\sum_{n=0}^{\infty} s_n(\mu) z^{-n-1}=\int_{\mathbb{R}}\frac{1}{x-z}\,\mu(dx).
\]
Then $M_{\mu}(z)$, $D_{\mu}(z)$ are defined on some some neighbourhood of $z=0$ and on $\mathbb{C}\setminus\mathbb{R}$, $G_{\mu}(z)$ is defined on some neighbourhood of $\infty$ and on $\mathbb{C}\setminus\mathbb{R}$. Moreover, $G_{\mu}$ maps $\mathbb{C}^{-}:=\{z\in\mathbb{C}:\mathrm{Im}z<0\}$ into $\mathbb{C}^{+}:=\{z\in\mathbb{C}:\mathrm{Im}z>0\}$ and $\mathbb{C}^{+}$ into $\mathbb{C}^{-}$, $D_{\mu}$ maps $\mathbb{C}^{+}$ into $\mathbb{C}^{+}$ and $\mathbb{C}^{-}$ into $\mathbb{C}^{-}$, which means that $D_{\mu}$ is a \textit{Pick function}. We have also $M_{\mu}(0)=1$, $D_{\mu}(0)=0$, $D_{\mu}'(0)=1$, $D_{\mu}(z)=zM_{\mu}(z)=G_{\mu}(1/z)$.
In addition, we define the \textit{free $R$-transform} of $\mu$ by the equation
\begin{equation}\label{eq:freerm}
1+R_{\mu}\left(zM_{\mu}(z)\right)=M_{\mu}(z)
\end{equation}
for $z$ in a neighbourhood of $0$, equivalently,
\begin{equation}
1+R_{\mu}\left(D_{\mu}(z)\right)=\frac{D_{\mu}(z)}{z}.
\end{equation}
The coefficients $r_n(\mu)$ in the Taylor expansion
\[
R_{\mu}(w)=\sum_{n=1}^{\infty} r_n(\mu) w^n
\]
are called the \textit{free cumulants} of $\mu$.
If $F_{\mu}(w)$ is the composition inverse of $D_{\mu}(z)$ on a neighbourhood of $z=0$ then
\begin{equation}
R_{\mu}(w)=\frac{w}{F_{\mu}(w)}-1,
\end{equation}
we refer to the monographs \cite{VoiDyNi1992,NicaSpeicher2006,MingoSpeicher2017} for more details.

The ``free part'' of the following theorem is a consequence of the theory of free probability, due to Voiculescu, see \cite{VoiDyNi1992,NicaSpeicher2006,MingoSpeicher2017}. The ``monotone part'' comes from the fact that the composition of two Pick functions is again a Pick function, which plays an important role in the theory of monotone independence due to Muraki~\cite{Muraki2001}.

\begin{theorem}
Assume that $\mu,\mu_1,\mu_2\in\mathcal{M}_{\mathrm{c}}$, $t\ge1$. Then there exist unique distributions in $\mathcal{M}_{\mathrm{c}}$, denoted $\mu^{\boxplus t},\mu_1\boxplus\mu_2,\mu_1\lhd\mu_2$, called respectively the additive free power, the additive free convolution and the monotone convolution, such that:
\begin{align*}
R_{\mu^{\boxplus t}}(w)&=t\cdot R_{\mu}(w),\\
R_{\mu_1\boxplus\mu_2}(w)&=R_{\mu_1}(w)+R_{\mu_2}(w),\\
D_{\mu_1\lhd\mu_2}(z)&=D_{\mu_1}\left(D_{\mu_2}(z)\right).
\end{align*}
\end{theorem}

Applying this theorem to our framework we obtain:

\begin{corollary}\label{cor:freemonotoneconv}
Assume that $F,F_1,F_2\in\mathcal{F}_{\mathrm{dist}}$ and $t\ge1$.
Then
\[
F^{\boxplus t},\,\,F_1\boxplus F_2,\,\, F_2\circ F_1\in\mathcal{F}_{\mathrm{dist}}.
\]
In particular, $(\mathcal{F}_{\mathrm{dist}},\boxplus,w)$ and $(\mathcal{F}_{\mathrm{dist}},\circ,w)$ constitute unital semigroups.
\end{corollary}

One can also combine free power with dilation.

\begin{corollary}\label{cor:rdilation}
Assume that $F\in\mathcal{F}_{\mathrm{dist}}$, with the corresponding $R$-transform $R(w)$, and that $t\ge1$, $\tau\in\mathbb{R}\setminus\{0\}$, and let $F_1\in\mathcal{F}$ correspond to the $R$-transform $R_1(w)=t\cdot R(\tau w)$. Then $F_1\in\mathcal{F}_{\mathrm{dist}}$.
\end{corollary}

\begin{example}\label{ex:rnn1}
Is $r_n=n$ the free cumulant sequence for a certain probability distribution?
We have
\[
R(w):=\sum_{n=1}^{\infty}n w^n=\frac{w}{(1-w)^2},
\]
and the corresponding $F$-function is
\[
F(w)=\frac{w(1-w)^2}{1-w+w^2}.
\]
However the sequence $s_n(F)$:
\[
1, 1, 3, 10, 37, 146, 602, 2563, 11181, 49720, 224540,\ldots,
\]
indexed in OEIS as A109081, is not positive definite, because $\det\big(s_{i+j}(F)\big)_{i,j=0}^{5}=-3374<0$.
Hence $r_n=n$ is not the free cumulant sequence of a distribution. One can also verify that the sequences $s_n\left(F^{\boxplus2}\right)$ and $s_n\left(F^{\boxplus3}\right)$ are not positive definite because $\det\left(s_{i+j}(F^{\boxplus2})\right)_{i,j=0}^{10}<0$
and $\det\left(s_{i+j}(F^{\boxplus3})\right)_{i,j=0}^{31}<0$.

The characteristic polynomial for $F^{\boxplus t}$ is
\[
\chi_t(w)=(1-w)(1-3w+3w^2-2tw^2-w^3),
\]
solving: $\chi_{t}(w)=\chi_{t}'(w)=0$, we get solutions: $(w,t)=(1,0)$ and $(w,t)=(-2,27/8)$. For $t=27/8$ we have
\[
\chi_{27/8}(w)=\frac{1}{4}(1-w)(1-4w)(2+w)^2.
\]
Therefore $F^{\boxplus 27/8}\in\mathcal{F}_{\mathrm{rr}}^{0}$ and the sequence $r_n=27n/8$ is the free cumulant sequence for a certain probability distribution.
We conjecture that $27/8$ is the smallest constant with this property.
\end{example}


\section{Freely infinitely divisible distributions in $\mathcal{M}_{\mathrm{ratio}}$}\label{sec:inf}

A distribution $\mu\in\mathcal{M}_{\mathrm{c}}$ is called \textit{freely infinitely divisible} if for every $n\in\mathbb{N}$ there exists $\mu_{n}\in\mathcal{M}_{\mathrm{c}}$ such that
$\mu=\mu_n^{\boxplus n}$. This implies, that for every real $t>0$ there exists a distribution, denoted $\mu^{\boxplus t}$, whose the $R$-transform is $t\cdot R_{\mu}(w)$. Denote by $\mathcal{M}_{\mathrm{ratio}}^{\mathrm{fid}}$ the class of those $\mu\in\mathcal{M}_{\mathrm{ratio}}$ which are freely infinitely divisible and by $\mathcal{F}_{\mathrm{dist}}^{\mathrm{fid}}$ the class of corresponding $F\in\mathcal{F}$.

Three important families of distributions belong to $\mathcal{M}_{\mathrm{ratio}}^{\mathrm{fid}}$, namely,
the one-point distribution $\delta_u$, $u\in\mathbb{R}$,
the Marchenko-Pastur law:
\[
\mathrm{MP}(v,t):=\max\{1-t,0\}\delta_0+\frac{1}{|v|}f_t\left(\frac{x}{v}\right)\,dx,
\]
$v\in\mathbb{R}\setminus\{\emptyset\}$, $t>0$,
where
\[
f_t(x)=
\left\{\begin{array}{ll}
\frac{\sqrt{4t-(x-1-t)^2}}{2\pi x}&\hbox{if $(x-1-t)^2<4t$},\\
0&\hbox{otherwise},
\end{array}\right.
\]
and the semi-circle Wigner law:
\[
\mathrm{W}(t)=
\frac{1}{2\pi t}\sqrt{4t-x^2}\mathbf{1}_{[-2\sqrt{t},2\sqrt{t}]}(x)\,dx,
\]
$t>0$.
The corresponding $R$-transforms and $F$-functions are the following:
\begin{align*}
R_{\delta_u}(w)&=uw, &F_{\delta_u}(w)&=\frac{w}{1+uw},\\
R_{\mathrm{MP}(v,t)}(w)&=\frac{tvw}{1-vw},
&F_{\mathrm{MP}(v,t)}(w)&=\frac{w(1-vw)}{1-vw+tvw},\\
R_{\mathrm{W}(t)}(w)&=tw^2, &F_{\mathrm{W}(t)}(w)&=\frac{w}{1+tw^2},
\end{align*}
and the characteristic polynomials are:
\begin{align*}
\chi_{\delta_u}(w)&=1,\\
\chi_{\mathrm{MP}(v,t)}(w)&=1-2vw+(1-t)v^2 w^2,\\
\chi_{\mathrm{W}(t)}(w)&=1-t w^2.
\end{align*}
All these distributions belong to $\mathcal{M}_{\mathrm{rr}}^{0}$.
The corresponding moment sequences are:
$s_n(\delta_u)=u^n$,
\[
s_n\big(\mathrm{MP}(v,t)\big)=v^n\sum_{k=1}^{n}\binom{n}{k-1}\binom{n-1}{k-1}\frac{t^k}{k}
\]
for $n\ge1$ (the coefficients $\binom{n}{k-1}\binom{n-1}{k-1}\frac{1}{k}$ are called \textit{Narayana numbers}, $A001263$ in OEIS), and
$s_{2n}\big(\mathrm{W}(t)\big)=\binom{2n+1}{n}\frac{t^n}{2n+1}$, $s_{2n+1}\big(\mathrm{W}(t)\big)=0$, with the Catalan numbers $\binom{2n+1}{n}\frac{1}{2n+1}$, $A000108$, as coefficients.

Now we will provide a version of the L\'evy--Khintchine theorem for the class $\mathcal{M}_{\mathrm{ratio}}^{\mathrm{fid}}$.

\newpage
\begin{theorem}\label{thm:levy}
A distribution $\mu\in\mathcal{M}_{\mathrm{c}}$ belongs to $\mathcal{M}_{\mathrm{ratio}}^{\mathrm{fid}}$ if and only if there exist some $u\in\mathbb{R}$, $c_0\ge0$, $r\ge0$, $c_1,\ldots,c_r>0$ and some distinct $u_1,\ldots,u_r\in\mathbb{R}\setminus\{0\}$ such that
\[
\mu=\delta_u\boxplus\mathrm{W}(c_0)\boxplus\mathrm{MP}(u_1,c_1)\boxplus\ldots
\boxplus\mathrm{MP}(u_r,c_r),
\]
so that
\begin{equation}\label{for:rinfdiv}
R_{\mu}(w)=uw+c_0 w^2 +\sum_{k=1}^{r}\frac{c_k u_k w}{1-u_k w}.
\end{equation}
Such a decomposition is unique.
\end{theorem}

\begin{proof}
A compactly supported distribution $\mu$ is freely infinitely divisible if and only if the
sequence $\left\{r_{n+2}(\mu)\right\}_{n=2}^{\infty}$ is positive definite,
where $r_n(\mu)$ are the \textit{free cumulants} of $\mu$, i.e. $R_{\mu}(w)=\sum_{k=1}^{\infty} r_{k}(\mu)w^k$. Therefore $R_{\mu}(w)/w$ is a Pick function, so $R_{\mu}(w)$ must be of the form (\ref{for:rinfdiv})
for some $u\in\mathbb{R}$, $c_0\ge0$, $r\ge0$, $c_1,\ldots,c_r>0$ and some distinct $u_1,\ldots,u_r\in\mathbb{R}\setminus\{0\}$, see~\cite{donoghue}.
\end{proof}

For $c_0=0$ the symbol $\mathrm{W}(0)$ denotes $\delta_0$, a term which can be ignored.
Now we will see that the class $\mathcal{F}_{\mathrm{dist}}^{\mathrm{fid}}$ can be characterized in terms of the characteristic polynomials.

\begin{theorem}\label{thm:infdivstars}
Assume that $F\in\mathcal{F}$.
Then $F\in\mathcal{F}_{\mathrm{dist}}^{\mathrm{fid}}$ if and only if there exists $t_0>0$
such that $F^{\boxplus t}\in\mathcal{F}_{\mathrm{rr}}^{0}$ for every $t\in(0,t_0)$.
\end{theorem}

\begin{proof}
We can assume that $u=0$ (due to  Proposition~\ref{prop:charpoltranslation}) and that $r\ge1$.
Then
\begin{multline*}
R(w)=c_0 w^2+\sum_{k=1}^{r}\frac{c_k u_k w}{1-u_k w}\\
=\frac{c_0 w^2\prod_{k=1}^{r}(1-u_k w)+\sum_{k=1}^{r} c_k u_k w\prod_{i\ne k}(1-u_i w)}{\prod_{k=1}^{r}(1-u_k w)}
\end{multline*}
for some $c_0\ge0$, $r\ge1$, $c_1,\ldots,c_r>0$ and $u_1,\ldots,u_r\in\mathbb{R}\setminus\{0\}$.
Then
\[
1+R(w)=\frac{(c_0 w^2+1)\prod_{k=1}^{r}(1-u_k w)+\sum_{k=1}^{r} c_k u_k w\prod_{i\ne k}(1-u_i w)}{\prod_{k=1}^{r}(1-u_k w)}
\]
and
\[
F(w)=\frac{wP(w)}{Q(w)},
\]
where
\begin{align*}
P(w)&=\prod_{k=1}^{r}(1-u_k w),\\
Q(w)&=(c_0 w^2+1)\prod_{k=1}^{r}(1-u_k w)+\sum_{k=1}^{r} c_k u_k w\prod_{i\ne k}(1-u_i w).
\end{align*}
Hence the characteristic polynomial $\chi_t(w)$ of $F^{\boxplus t}$ is
\[
\chi_t(w)=P(w)^2+t B(w),
\]
where
\begin{multline*}
B(w):=\left(\prod_{k=1}^{r}(1-u_k w)-w\sum_{k=1}^{r}u_k\prod_{i\ne k}(1-u_i w)\right)\big(Q(w)-P(w)\big)\\
-w \prod_{k=1}^{r}(1-u_k w)\big(Q'(w)-P'(w)\big),
\end{multline*}
and
\[
Q(w)-P(w)=
c_0 w^2\prod_{k=1}^{r}(1-u_k w)+\sum_{k=1}^{r} c_k u_k w\prod_{i\ne k}(1-u_i w),
\]
\begin{multline*}
Q'(w)-P'(w)=2c_0 w\prod_{k=1}^{r}(1-u_k w)\\
+\sum_{k=1}^{r}u_k \left( (c_k-c_0 w^2)\prod_{i\ne k}(1-u_i w)-c_k w \sum_{i\ne k}u_i\prod_{j\ne k,i}(1-u_j w)\right).
\end{multline*}

The leading term of $wP(w)$ is $w^{r+1}\prod_{k=1}^{r}(-u_k)$. If $c_0>0$
then the leading term of $Q(w)-P(w)$ is $c_0 w^{r+2}\prod_{k=1}^{r}(-u_k)$.
In this case the leading term of $B(w)$ is
\[-w^{2r+2}c_0\prod_{k=1}^{r}u_k^2.\]
If $c_0=0$ then the leading term of $Q(w)-P(w)$ is $-w^r c\prod_{k=1}^{r}(-u_k)$, where $c:=c_1+\ldots+c_r$, and the leading term of $B(w)$
is \[-w^{2r}c\prod_{k=1}^{r}u_k^2.\]

Now observe that $P(1/u_k)=0$, and
\[
B\left(\frac{1}{u_k}\right)=-c_k \prod_{i\ne k}\left(1-\frac{u_i}{u_k}\right)^2<0
\]
for $k=1,\ldots,r$.
Put $v_k:=1/u_k$ and assume that $v_1<v_2<\ldots<v_r$. Let $x_k$
be such that $P'(x_k)=0$ and $v_k<x_k<v_{k+1}$, $1\le k<r$. In addition we set $x_0:=v_1-1$, $x_r:=v_r+1$, so that
\[x_0<v_1<x_1<v_2<\ldots<x_{r-1}<v_{r}<x_r.
\]

Now let us choose $t_0>0$ such that $\chi_{t_0}(x_k)=P(x_k)^2+t_0 B(x_k)>0$ for $k=0,\ldots,r$, and fix $t$ such that $0<t<t_0$.
Then $\chi_t(x_k)>0$, $\chi_{t}(v_k)<0$, hence there are $y_k',y_k''$ such that
\[
x_{k-1}<y_k'<v_k<y_k''<x_k\quad
\hbox{and}\quad \chi_t(y_k')=\chi_{t}(y_k'')=0
\]
for $k=1,\ldots,r$.
If, in addition, $c_0>0$ then $\lim_{w\to\pm\infty} \chi_{t_0}(w)=-\infty$ and hence there are $y_{-}<x_0$, $y_{+}>x_r$
such that $\chi_{t}(y_{-})=\chi_{t}(y_{+})=0$. Therefore all the roots of $\chi_{t}(w)$ are real.
\end{proof}

Here is an obvious consequence from the last proof.

\begin{corollary}\label{cor:fidrealsimple}
If $F\in\mathcal{F}_{\mathrm{dist}}$, $F(w)=wP(w)/Q(w)$, and the corresponding distribution is freely infinitely divisible then all the roots of $P(w)$ are real and simple.
\end{corollary}

The following example illustrates Theorem~\ref{thm:infdivstars}.

\begin{example}[$A078623$]  
Take
\[
\mu:=\mathrm{MP}(1,1)\boxplus\mathrm{W}(1)\boxplus\delta_{-1},
\]
with the moment sequence $A078623$, so that
\[
F(w)=\frac{w(1-w)}{1-w+2w^2-w^3}.
\]
Then $\mu\in\mathcal{M}_{\mathrm{ratio}}^{\mathrm{fid}}$, but one can check, that $F\notin\mathcal{F}_{\mathrm{rr}}$.
We have
\[
\chi_{t}(w)=1-2w+w^2-2tw^2+2tw^3-tw^4.
\]
Substituting $t:=(x-1)^2/(4x)$ we obtain
\[
\chi_{t}(w)=\frac{1}{4x}(2x-2xw-w^2+xw^2)(2-2w+w^2-xw^2).
\]
Note that the function $x\mapsto(x-1)^2/(4x)$ is increasing on $(1,+\infty)$
and maps $(1,+\infty)$ onto $(0,+\infty)$. For $x>1$ the polynomial $\chi_{t}(w)$
has only real roots if and only if $1<x\le2$, equivalently, if an only if $0<t\le1/8$.
Therefore, $F^{\boxplus t}\in\mathcal{F}_{\mathrm{rr}}^{0}$ if and only if $0<t\le1/8$.
In particular
\[
\chi_{1/8}(w)=\frac{1}{8}(2-w)^2(2-2w-w^2).
\]
\end{example}

\section{Examples of free deconvolution}\label{sec:deconv}

If $\mu,\mu_1,\mu_2\in\mathcal{M}_{\mathrm{c}}$ and $\mu=\mu_1\boxplus\mu_2$ then we
say that $\mu_2$ is \textit{free deconvolution} of $\mu$ and $\mu_1$ and write $\mu_2=\mu\boxminus\mu_1$. If in addition $\mu,\mu_1$ are freely infinitely divisible
then we say that $\mu_2$ is \textit{freely quasi infinitely divisible}, see \cite{hmsu2021} for thorough studies of this notion.
Here we are going to examine two families of free deconvolutions by looking for singular elements.
Other examples, coming from monotone convolutions, will be found in Section~\ref{sec:monotone}.

\subsection{Wigner law and Marchenko-Pastur law}\label{subsec:deconv:wmp}

\begin{proposition}\label{pro:deconvolutionwmp}
For $u,x\in\mathbb{R}$ there is a probability distribution $\mu$ such that
\[
R_{\mu}(w)=u^2 x^3 w^2+\frac{(1-x)^3 uw}{1-uw}.
\]
\end{proposition}

We can write symbolically
\[
\mu=\mathrm{W}\left(u^2 x^3\right)\boxplus\mathrm{MP}\left(u, (1-x)^3\right),
\]
for example $\mu=\mathrm{W}(-1)\boxplus\mathrm{MP}(1,8)$ means that $\mu\boxplus\mathrm{W}(1)=\mathrm{MP}(1,8)$.
Note that $\mu$ is freely infinitely divisible if and only if $0\le x\le1$.

\begin{proof}
Take
\[
R(w)=aw^2+\frac{buw}{1-uw},
\]
so that
\[
F(w)=\frac{w(1-uw)}{1-uw+buw+aw^2-au w^3},
\]
\[
\chi_{F}(w)=1 - 2 u w - a w^2 + u^2 w^2 - b u^2 w^2 + 2 a u w^3 - a u^2 w^4.
\]
Putting
\begin{equation}
a:=u^2 x^3,\qquad b:=(1-x)^3
\end{equation}
we get
\[
\chi_{F}(w)=(1-u x w)^2 (1 - 2 u w + 2 u x w - u^2 x w^2),
\]
so all the roots of $\chi_{F}(w)$ are real.
\end{proof}

\begin{example}[$A007297$]
Taking
\[
F(w)=\frac{w(1-w)}{(1+w)^3},
\]
we get sequence $A007297$ as $s_n(F)$. Then
\[
R(w)=-4w-w^2+\frac{8w}{1-w},
\]
$\chi_{F}(w)=(1+w)^2(1-4w+w^2)$, so $s_n(F)$ is positive definite
as the moment sequence of a distribution $\mu$, which can be symbolically denoted as
\[
\mu=\delta_{-4}\boxplus\mathrm{W}(-1)\boxplus\mathrm{MP}(1,8),
\]
translation of $\mathrm{W}(-1)\boxplus\mathrm{MP}(1,8)$ by $-4$.
\end{example}

\subsection{Two Marchenko-Pastur laws}\label{subsec:deconv:mpmp}

\begin{proposition}\label{pro:deconvolutionmpmp}
For $u,v,x\in\mathbb{R}$, with $u,v\ne0$, $u\ne v$, there exists a probability distribution $\mu$ such that
\[
R_{\mu}(w)=\frac{(u-x)^3}{u^2(u-v)}\frac{uw}{1-uw}+
\frac{(v-x)^3}{v^2(v-u)}\frac{vw}{1-vw}.
\]
\end{proposition}

We can write symbolically
\[
\mu=\mathrm{MP}\left(u,\frac{(u-x)^3}{u^2(u-v)}\right)\boxplus\mathrm{MP}\left(v,\frac{(v-x)^3}{v^2(v-u)}\right).
\]
If $u<v$ then $\mu$ is freely infinitely divisible if and only if $u\le x\le v$.

\begin{proof}
Take
\[
R(w)=\frac{auw}{1-uw}+\frac{bvw}{1-vw},
\]
so that
\[
F(w)=\frac{w(1-uw)(1-vw)}{1+(a+b-u-v)w-(av+bu)w^2+uvw^2}.
\]
The characteristic polynomial is
\begin{multline*}
\chi_{F}(w)=1-2(u+v)w+(u^2+v^2+4uv-au^2-bv^2)w^2\\
-2uv(u+v-au-bv)w^3+u^2v^2(1-a-b)w^4.
\end{multline*}
If we make substitution:
\begin{equation}\label{eq:weightsab}
a:=\frac{(u-x)^3}{u^2(u-v)},\qquad
b:=\frac{(v-x)^3}{v^2(v-u)},
\end{equation}
then
\[
\chi_{F}(w)=(1-xw)^2\left(1+2(x-u-v)w+(3uv-xu-xv)w^2\right),
\]
and all the roots are real.
\end{proof}

Note that in both propositions, \ref{pro:deconvolutionwmp} and \ref{pro:deconvolutionmpmp}, the parameters can be chosen in such a way that the sum of weights $a+b$ is negative.

\begin{example}[$A085614$, $A250886$]
If we take $x=0,u=1,v=2$ then \[F(w)=w(1-w)(1-2w)\] and, symbolically, $\mu=\mathrm{MP}(1,-1)\boxplus\mathrm{MP}(2,2)$,
with the moment sequence $a_n=\frac{2^n\cdot(3n)!!}{(n+1)! n!!}$, indexed as $A085614$ in OEIS.
Note also, that taking $x=0,u=-1,v=2$ we get $\mu=\mathrm{MP}(1/3,-1)\boxplus\mathrm{MP}(2/3,2)$,
a freely infinitely divisible distribution, with $A250886$ as the moment sequence.
\end{example}

\section{Polynomials as the $R$-transforms}\label{sec:rtransform}

Polynomials which are free $R$-transforms of probability distributions were studied by
Chistyakov and G\"otze~\cite{chistyakov2011}, who proved the following theorem:

\begin{theorem}
Let $R(w)$ be a polynomial with real coefficients, with $R(0)=0$, and let $\Omega$
denote the component of
\[
\left\{w\in\mathbb{C}^{+}:\mathrm{Im}\big(w+wR(1/w)\big)>0\right\}
\]
which contains $\infty$,
Then $R(w)$ is the free $R$-transform of some probability distribution if and only if every Jordan curve, contained in $\mathbb{C}^{+}\cup\mathbb{R}$ and connecting $0$ and $\infty$, contains a point from the boundary of $\Omega$.
\end{theorem}

In this section we will assume that $R(w)$ is a polynomial, with $R(0)=0$,
$F$ is the corresponding element of $\mathcal{F}$, so that $F(w)=w/\big(1+R(w)\big)$
and $F\in\mathcal{F}(0,q)$ for some $q\ge1$.
Then the characteristic system (\ref{eq:charsystem}) of equations reduces to
\begin{equation}\label{eq:charsystemrpoly}
\begin{split}
w&=z\left(1+R(w)\right),\\
1&=z R'(w).
\end{split}
\end{equation}
and the characteristic polynomial is
\[
\chi_{F}(w)=1+R(w)-wR'(w).
\]
We assume, with no loss of generality, that $R'(0)=0$.
Moreover, if $R(w)$ is the free $R$-transform of a probability distribution $\mu$ then $R''(0)/2$ is the variance of $\mu$, so we can assume that $R''(0)>0$.
We will study conditions for which $F\in\mathcal{F}_{\mathrm{rr}}^{0}$. We will see that if $\deg R(w)=3$ then $F\in\mathcal{F}_{\mathrm{rr}}^{0}$ if and only if $F\in\mathcal{F}_{\mathrm{dist}}$, however if $\deg R(w)=4$ then it is not the case, as one can see on Figure~\ref{fig:galaxy}.

\subsection{Special case: $\deg R(w)=3$}\label{subsec:r3}

\begin{example}
If $R(w)=3w^2+w^3$ then $\chi_{F}(w)=(1+w)^2(1-2w)$ and $F(w)=w/(1+3w^2+w^3)\in\mathcal{F}_{\mathrm{rr}}^{0}$.
Hence there is a distribution $\mu$ on $\mathbb{R}$ such that $F=F_{\mu}$.
The corresponding sequence $s_n(F)$ is $A120984$.
More generally, taking $R(w)=3uw+3w^2+w^3$, the corresponding distribution $\mu_{u}$ is translation of $\mu$ by $3u$
and the coefficients of the moments of $\mu_{u}$ form array $A120981$.
\end{example}

\begin{proposition}\label{prop:rpolynomial3}
Assume that $R(w)=uw+bw^2+cw^3$, $b>0$. Then $F\in\mathcal{F}_{\mathrm{rr}}^{0}$ if and only if $27c^2\le b^3$.
\end{proposition}

It follows from Theorem~\ref{thm:cg} that the condition $27c^2\le b^3$ is also necessary for $F$ to belong to $\mathcal{F}_{\mathrm{dist}}$, so that the sets $\mathcal{F}_{\mathrm{dist}}\cap\mathcal{F}(0,3)$ and $\mathcal{F}_{\mathrm{rr}}^{0}\cap\mathcal{F}(0,3)$ coincide.

\begin{proof} We have $\chi_{F}(w)=1-bw^2-2cw^3$.
If $F\in\mathcal{F}_{\mathrm{rr}}^{0}$ then
\[\chi(w)
=(1-v_1 w)(1-v_2 w)(1-v_3 w)\]
for some $v_1,v_2,v_3\in\mathbb{R}$.
Put
\begin{align*}
g_1(v_1,v_2,v_3)&:=v_1+v_2+v_3,\\
g_2(v_1,v_2,v_3)&:=v_1 v_2+v_1 v_3+v_2 v_3+b,\\
f(v_1,v_2,v_3)&:=v_2 v_2 v_3/2.
\end{align*}
Then
\[\det\left(\nabla g_1,\nabla g_2,\nabla f\right)=\frac{1}{2}(v_1-v_2)(v_1-v_3)(v_2-v_3)
\]
and hence it is easy to see that under the condition $g_1(v_1,v_2,v_3)=g_2(v_1,v_2,v_3)=0$ the range of $f$ is $\left[-\sqrt{b^3/27},\sqrt{b^3/27}\right]$.
\end{proof}

For example, taking $F(w)=\frac{w}{(1+w)^2(1+2w)}$ we have $R(w)=4w+5w^2+2w^3$, so the corresponding sequence $s_{n}(F)$,
denoted $A003168$ in OEIS, is positive definite.
Another example, taking $F(w)=\frac{w}{(1+w)(1+w+w^2)}$ we obtain $A106228$ as $s_n(F)$.
This sequence is not positive definite: $R(w)=2w+2w^2+w^3$ and $\det\left(s_{i+j}(F)\right)_{i,j=0}^{5}=-3374<0$.

\subsection{Special case: $\deg R(w)=4$}\label{subsec:r4}

Assume that
\begin{equation}
R(w)=b w^2 +c w^3+dw^4,
 \end{equation}\label{eq:rtransorm4}
 $F(w)=w/(1+R(w))$. Then
\[
\chi_{F}(w)=1-b w^2-2 c w^3-3 d w^4
\]
and
\[
\chi_{F}'(w)=-2w\left(b+3c w+6d w^2\right).
\]
Therefore a necessary condition for $F$ to belong to $\mathcal{F}_{\mathrm{rr}}^{0}$ is
\begin{equation}
3c^2\ge 8b d.
\end{equation}

We will see that on the
borderline of this class there are some elements in $\mathcal{F}_{\mathrm{dist}}$.

\begin{example}
Assume that $3c^2=8bd$ and put $t_0:=27c^2/(2b^3)$. Then
\[
\chi_{t_0}(w)=\frac{1}{16b^4}\left(2b+3cw\right)^3\left(2b-9cw\right).
\]
Hence if $27c^2\le2b^3$  (equivalently: $36d\le b^2$) then $F\in\mathcal{F}_{\mathrm{dist}}$.
\end{example}

\subsubsection{Special subcase: $c=0$}

Let us confine ourselves for a moment to the case $c=0$, so that $R(w)=bw^2+dw^4$ and $\chi_{F}(w)=1-bw^2-3dw^4$. Then $F\in\mathcal{F}_{\mathrm{rr}}^{0}$ if and only if $-b^2\le12b\le0$ and if $-b^2<12b<0$ then all the roots of $\chi(w)$ are simple.

If $4d=b^2$ then
\[
Q(w)=1+R(w)=\left(1+\frac{b}{2}w^2\right)^2,
\]
\[
\chi_{F}(w)=\frac{1}{4}\left(2+bw^2\right)\left(2-3bw^2\right).
\]
The solutions of the characteristic system (\ref{eq:charsystem}) are
\[
(w,z)=\left(\pm\sqrt{\frac{2}{3b}},\pm\sqrt{\frac{27}{128b}}\right)
\]
so $F\in\mathcal{F}_{\mathrm{rr}}$. The point $(c,d)=(0,1/4)$ is exactly the top of the ``baloon'' in Figure~\ref{fig:galaxy}.

In this case we are able to prove the following proposition:

\begin{proposition}\label{pr:r4c0}
Let $R(w)=bw^2+dw^4$, $b>0$, and let $F(w)=w/(1+R(w))$ be the corresponding element of $\mathcal{F}$. Then
\begin{itemize}
    \item if $-b^2\le12d\le3b^2$ then $F\in\mathcal{F}_{\mathrm{dist}}$,
    \item $F\in\mathcal{F}_{\mathrm{rr}}^{0}$ if and only if $-b^2\le12d\le0$,
    \item if $4d=b^2$ then $F\in\mathcal{F}_{\mathrm{rr}}\setminus\mathcal{F}_{\mathrm{rr}}^{0}$.
\end{itemize}
\end{proposition}

\begin{proof}
We need to prove only the first statement. For $R(w)=bw^2+bw^4/4$ and $0<t\le1$
take \[R_1(w):=\frac{1}{t^2}R(tw)=bw^2 +\frac{t^2 d^2}{4}w^4.\]
Now we conclude by applying Corollary~\ref{cor:rdilation}.
\end{proof}

\subsubsection{Arbitrary $c$}
Now we come back to arbitrary real $c$.
First we describe the singular elements in this class.

\begin{proposition}\label{pr:r4sing}
For $R(w)=b w^2 +c w^3+dw^4$ and for $F(w)=w/(1+R(w))$ the characteristic function $\chi_{F}(w)$ has a multiple real root if and only if
\[
c=v(b-2v^2),\qquad
d=\frac{1}{3}v^2(b-3v^2)
\]
for some $v\in\mathbb{R}$. In this case
\[
\chi_{F}(w)=(1+vw)^2\left(1-2vw+3v^2 w^2-bw^2\right).
\]
Such $F$ belongs to $\mathcal{F}_{\mathrm{rr}}^{0}$ if and only if $2v^2\le b$.
\end{proposition}

\begin{proof}
Solving the system of equations: $\chi(w_0)=0$, $\chi'(w_0)=0$ with respect to $c,d$ we get
\[
c=\frac{2-bw_0^2}{w_0^3},\qquad
d=\frac{bw_0^2-3}{3w_0^4}.
\]
Then substitute $w_0:=-1/v$. The rest of the proof is straightforward.
\end{proof}

Now we are able to describe the elements of $\mathcal{F}_{\mathrm{rr}}^{0}$.

\begin{proposition}
For fixed $b>0$ denote by $\mathcal{L}(b)$ the set of all $(c,d)\in\mathbb{R}^2$ such that
\[
c=v(b-2v^2),\qquad
d=\frac{1}{3}v^2(b-3v^2)
\]
for some $v\in\mathbb{R}$ such that $2v^2\le b$.
Then $\mathcal{L}(b)$ is a Jordan curve in $\mathbb{R}^2$.

For $R(w)=b w^2 +c w^3+dw^4$, with $b>0$ fixed, and for $F(w)=w/(1+R(w))$
we have $F\in\mathcal{F}_{\mathrm{rr}}^{0}$ if and only if $(c,d)$ belongs either to $\mathcal{L}(b)$
or to the interior region of $\mathcal{L}(b)$.
\end{proposition}

\begin{proof}
Assume that
\[
v_1(b-2v_1^2)=v_2(b-2v_2^2),\qquad
v_1^2(b-3v_1^2)=v_2^2(b-3v_2^2).
\]
Then
\[
(v_1-v_2)\left(b-2v_1^2-2v_1 v_2-2v_2^2\right)=0,\quad
\left(v_1^2-v_2^2\right)\left(b-3v_1^2-3v_2^2\right)=0.
\]
If $v_1\ne v_2$ then either $v_2=-v_1$ and then the first equation implies $2v_1^2=2v_2^2=b$, or
\[
b=2v_1^2+2v_1 v_2+2v_2^2\quad\hbox{and}\quad
b=3v_1^2+3v_2^2,
\]
but this in turn implies $v_1=v_2$. Therefore $\mathcal{L}(b)$ is a Jordan curve.

Observe that if $-b^2<12d<0$ then the point $(0,d)$ lies in the interior region of $\mathcal{L}(b)$,
so the second statement follows from Corollary~\ref{cor:topology}.
\end{proof}

Now we quote a result due to Chistyakov and G\"{o}tze \cite{chistyakov2011}:

\begin{theorem}\label{thm:cg}
A sequence $(\kappa_1,1,\kappa_3,\kappa_4,0,0,\ldots)$ is the free cumulant sequence of some probability measure if and only if $(\kappa_3,\kappa_4)\in\mathcal{D}_1\cup\mathcal{D}_2$, where
$\mathcal{D}_1$ consists of such $(x,y)\in\mathbb{R}^2$ that
\[
|x|\le\frac{1}{3\sqrt{6}}\sqrt{1+\sqrt{1-36y}}\left(2-\sqrt{1-36y}\right),
\qquad\frac{-1}{12}\le y\le\frac{1}{36},
\]
and $\mathcal{D}_2$ consists of such $(x,y)\in\mathbb{R}^2$ that
\[
|x|\le2\sqrt{6y\left(\sqrt{y}-2y\right)},\qquad\frac{1}{36}<y\le\frac{1}{4}.
\]
\end{theorem}

\begin{figure}[htp]
    \centering
    \includegraphics[width=15cm]{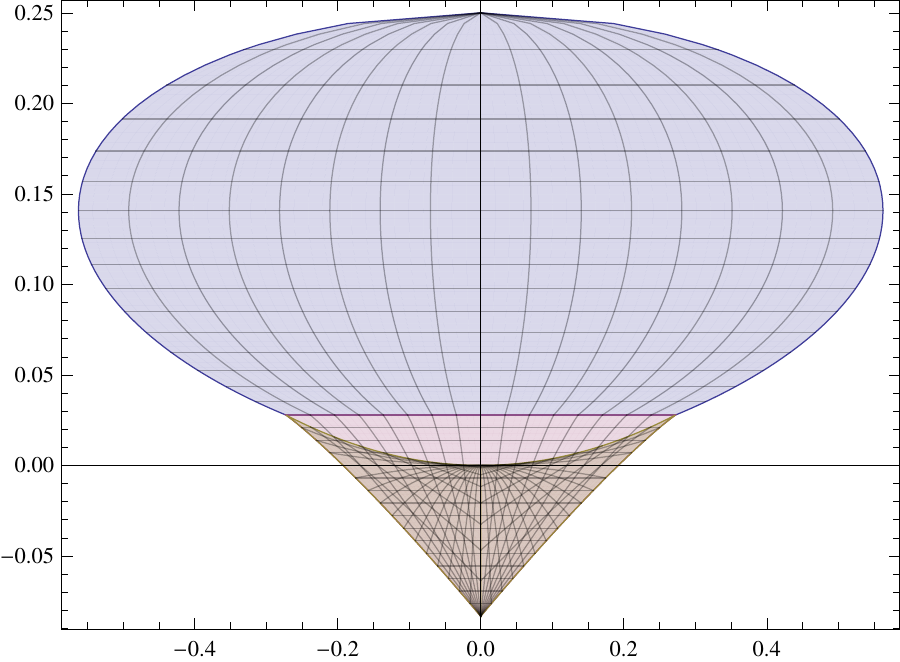}
    \caption{The sets $\mathcal{D}_1$, $\mathcal{D}_2$ and $\mathcal{L}(b)$ with the interior region (dark ``triangle'') for $b=1$}
    \label{fig:galaxy}
\end{figure}

Note that putting
\[
u:=\pm\sqrt{\frac{1+\sqrt{1-36y}}{6}}
\]
the points $(x,y)$ at the left and right border of $\mathcal{D}_1$ can be described as
\[
x=u(2u^2-1),\qquad
y=\frac{1}{3}u^2\left(1-3u^2\right),
\]
with ${1/6}\le u^2\le {1/2}$.


Now we will discuss a family of examples, which includes that corresponding to the top of the domain $\mathcal{D}_2$ on Figure~\ref{fig:galaxy}.

\begin{example}
Put
\begin{equation}\label{eq:rfussr}
R(w):=(1+w^2)^r-1,
\end{equation}
so that
\[
\chi(w)=\left(1+w^2-2rw^2\right)\left(1+w^2\right)^{r-1}
\]
(see Remark~\ref{remark:qmultiple}).
The solutions of (\ref{eq:charsystemrpoly}) are
\[
\left(\frac{\pm1}{\sqrt{2r-1}},\frac{\pm1}{\sqrt{2r-1}}\left(\frac{2r-1}{2r}\right)^{r}\right)
\]
for $r=1,2,\ldots$, so $F\in\mathcal{F}_{\mathrm{rr}}$.
Let $M(z)$ be the corresponding moment generating function. Then, by (\ref{eq:freerm}) and (\ref{eq:rfussr}),
\[
M(z)=\left(1+z^2 M(z)^2\right)^r.
\]
Put $M_1(z):=M(z)^{1/r}$, so that $M_1(z)$ satisfies
\[
M_1(z)=1+z^2 M_1(z)^{2r},
\]
which means, that $M_1(z)=\mathcal{B}_{2r}(z^2)$, where $\mathcal{B}_{p}(z)$ denotes the generating function for the Fuss numbers $\binom{np+1}{n}\frac{1}{np+1}$ (see \cite{concrete,mlotkowski2010doc}) and $M(z)=\mathcal{B}_{2r}(z^2)^{r}$.
From the Lambert formula, (5.60) in~\cite{concrete},
\begin{equation}\label{eq:rfussgf}
M(z)=\sum_{n=0}^{\infty}\binom{2nr+r}{n}\frac{r}{2nr+r}z^{2n}.
\end{equation}
Actually, a more general statement is true, namely for each real $r\ge1/2$ there is a symmetric probability distribution $\mu$ such that the $R$-transform $R_{\mu}(w)$ is given by (\ref{eq:rfussr}) and the moment generating function $M_{\mu}(w)$ is given by (\ref{eq:rfussgf}), c.f. formula~(5.4) in~\cite{mlotkowski2010doc}.
The densities of the corresponding distribution (more precisely, of the distribution $\nu$ on $[0,+\infty)$ which satisfies $\int_{\mathbb{R}}f(x^2)\mu(dx)=\int_{\mathbb{R}}f(x)\nu(dx)$ for any continuous function $f$) can be found in \cite{forrester2015,mpz2013}, see also \cite{liupego2016,MSU2020}.
Putting $r=1,2,3,4,5$ we get sequences $A000108$, $A069271$, $A212072$, $A234463$, $A234528$ respectively.
\end{example}

\section{Eulerian polynomials}\label{sec:euler}

The purpose of this section is to generalize Example~\ref{ex:rnn1}.
It is well known (see~\cite{petersen,concrete}) that for a given integer $k\ge1$ we have
\[
\sum_{n=1}^{\infty} n^k w^n=\frac{w E_k(w)}{(1-w)^{k+1}},
\]
where the \textit{Eulerian polynomials} are given by the following recurrence:
$E_0(w)=1$ and
\begin{equation}\label{eulerianpolynomials}
E_k(w)=\big(kw-w+1\big)E_{k-1}(w)+w(1-w)E'_{k-1}(w)
\end{equation}
for $k\ge1$.
We will need some auxiliary polynomials defined by
\begin{equation}
\widetilde{E}_k(w):=
(k+1)E_{k}(w)+(1-w)E_{k}'(w).
\end{equation}

\begin{lemma}
The polynomials $\widetilde{E}_k(w)$ satisfy the following recurrence:
$\widetilde{E}_0(w)=1$ and
\begin{equation}
\widetilde{E}_k(w)=(kw-w+2)\widetilde{E}_{k-1}(w)+w(1-w)\widetilde{E}'_{k-1}(w)
\end{equation}
for $k\ge1$.
\end{lemma}

\begin{proof}
The right hand side is
\begin{multline*}
\mathrm{RHS}:
=(kw-w+2)\left[k E_{k-1}(w)+(1-w)E'_{k-1}(w)\right]\\
+w(1-w)\left[k E'_{k-1}(w)-E'_{k-1}(w)+(1-w)E''_{k-1}(w)\right]\\
=k(kw-w+2)E_{k-1}+2(1-w)(kw-w+1)E'_{k-1}(w)+w(1-w)^2 E''_{k-1}(w).
\end{multline*}
From (\ref{eulerianpolynomials}) we can derive
\[
w(1-w)E''_{k-1}(w)=E'_{k}(w)-(k-1)E_{k-1}(w)-(kw-3w+2)E'_{k-1}(w),
\]
which leads to
\begin{multline*}
\mathrm{RHS}=(k+1)(kw-w+1)E_{k-1}(w)+(k+1)w(1-w)E'_{k-1}(w)
+(1-w)E'_{k}(w)\\
=(k+1)E_{k}(w)+(1-w)E'_{k}(w)=\widetilde{E}_{k}(w),
\end{multline*}
by definition.
\end{proof}

It is easy to see that $\deg(E_k)=\deg(\widetilde{E}_k)=k-1$ for $k\ge1$.
Denote by $e(k,i)$, $\widetilde{e}(k,i)$ the coefficients of the polynomials
$E_k(w)$, $\widetilde{E}_{k}(w)$:
\[
E_k(w)=\sum_{k=0}^{n}e(k,i) w^i,\qquad
\widetilde{E}_k(w)=\sum_{k=0}^{n}\widetilde{e}(k,i) w^i.
\]
Then $e(0,0)=\widetilde{e}(0,0)=1$,
$e(0,i)=\widetilde{e}(0,i)=0$ for $i\ne0$, and
\[
e(k,i)=(k-i)e(k-1,i-1)+(k+1)e(k-1,i),
\]
\[
\widetilde{e}(k,i)=(k-i)\widetilde{e}(k-1,i-1)+(k+2)\widetilde{e}(k-1,i)
\]
for $0\le i\le k$, $k\ge1$.
The coefficients $e(k,i)$ are the classical (type $A$) \textit{Eulerian numbers}, namely $e(k,i)$ is the number of such permutations $\sigma\in\mathcal{S}_k$ that $\sigma$ has $i$ descents,  see entry $A123125$ in OEIS. The numbers $\widetilde{e}(k,i)$ are less known and appear in the work of Conger \cite{conger2010,conger2007thesis},
namely $\widetilde{e}(k,i)$ is the number of such permutations $\sigma\in\mathcal{S}_{k+2}$ that $\sigma(1)=2$ and $\sigma$ has $i+1$ descents, see also \cite{kril2019} and entry $A120434$ in OEIS.
It is well known, and proved by Frobenius in 1910, that all the roots of $E_k(w)$ are real, negative, and simple. In view of Theorem~5 in \cite{conger2010}, the same is true for the polynomials $\widetilde{E}_k(w)$.

Now we are ready to prove the main theorem of this section.

\begin{theorem}
For every integer $k\ge1$ there is a constant $C>0$ such that $r_n:=C n^k$ is the free cumulant
sequence for a certain probability distribution.
\end{theorem}

\begin{proof}
Set
\[
R(w):=\sum_{n=1}^{\infty} n^k w^n=\frac{w E_k(w)}{(1-w)^{k+1}}.
\]
Then
\[
F(w)=\frac{w(1-w)^{k+1}}{(1-w)^{k+1}+w E_k(w)}
\]
and
\[
F^{\boxplus t}(w)=\frac{w(1-w)^{k+1}}{(1-w)^{k+1}+tw E_k(w)}.
\]
The characteristic polynomial for $F^{\boxplus t}$ is
\begin{multline*}
\chi_t(w)=(1-w)^{2k+2}-tw^2(1-w)^k\left[(k+1)E_{k}(w)+(1-w)E_{k}'(w)\right]\\
=(1-w)^{2k+2}-tw^2(1-w)^k \widetilde{E}_{k}(w)\\
=(1-w)^k\left((1-w)^{k+2}-tw^2\widetilde{E}_{k}(w)\right).
\end{multline*}
Now we are going to show, that if $t$ is large enough then $\chi_t(w)$
has only real roots, consequently, $F^{\boxplus t}\in\mathcal{F}_{\mathrm{rr}}^{0}$. It suffices to show, that for large $t>0$ the equation
\[
(1-w)^{k+2}=tw^2 \widetilde{E}_{k}(w)
\]
has $k+2$ real solutions.

Put $f(w):=(1-w)^{k+2}$, $g(w):=w^2 \widetilde{E}_{k}(w)$
and let $x_j$ and $y_j$ denote the roots of $g(w)$ and $g'(w)$, respectively, and
\[
x_1<y_1<x_2<y_2<\ldots<x_{k-1}<y_{k-1}<x_{k}=y_k=0.
\]

If $k=2r$ is even then
$g(w)$ has local maximum at $y_{2j-1}$, local minimum at $y_{2j}$, $1\le j\le r$, and $g(y_{2j-1})>0$. Choose $t>0$ such that $f(y_{2j-1})<t g(y_{2j-1})$, $1\le j\le r$. Then the functions $f(w)$ and $tg(w)$ cross at some point in $(x_{2j-1},y_{2j-1})$, and in $(y_{2j-1},x_{2j})$. Also they cross at some point in $(0,1)$.
Finally, since $\deg f>\deg g$, these functions cross at some point in $(1,+\infty)$, which completes the proof in this case.

Now assume that $k=2r+1$ is odd.
The polynomial $g(w)=w^2\widetilde{E}_k(w)$ has minimum at $y_1,y_3,\ldots,y_{k-2},y_{k}=0$ and maximum at $y_2,y_{4},\ldots,y_{2r}$. Put $y_0:=x_1-1$ and choose $t>0$ such that $f(y_{2j})<t\cdot g(y_{2j})$, $0\le j\le r$. Then the functions $f(w), t\cdot g(w)$ must cross at some point in each of the intervals $(x_{2j},y_{2j})$ and $(y_{2j},x_{2j+1})$, $j=1,2,\ldots,r$.
Moreover they must cross at some point in $(0,1)$. Also, they must cross in $(y_0,x_1)$. Finally, since $f(y_0)<t\cdot g(x_0)$ and $\deg f>\deg g$, these function must cross somewhere on $(-\infty,y_0)$. This completes the proof.
\end{proof}

Let $C_k$ denote the smallest constant such the $r_n:=C_k n^k$ is the free cumulant sequence for a certain probability distribution. In Example~\ref{ex:rnn1} we found that $3<C_1\le 27/8$. Now we will estimate $C_2$.

\begin{example}\label{ex:euler2}
For $r_n:=n^2$ we have
\[
R(w)=\sum_{n=1}^{\infty}n^2 w^n=\frac{w(1+w)}{(1-w)^3},
\]
so that
\[
F(w)=\frac{w(1-w)^3}{1-2w+4w^2-w^3}.
\]
The corresponding sequence $s_n(F)$:
\[
1, 1, 5, 22, 109, 576, 3174, 18047, 105093, 623608, 3757124,\ldots,
\]
is not positive definite: $\det\left(s_{i+j}(F)\right)_{i,j=0}^{4}=-685964$.
Moreover, we have also
\[
\det\left(s_{i+j}(F^{\boxplus 2})\right)_{i,j=0}^{5}<0,\quad
\det\left(s_{i+j}(F^{\boxplus 3})\right)_{i,j=0}^{7}<0,\quad
\det\left(s_{i+j}(F^{\boxplus 4})\right)_{i,j=0}^{10}<0,
\]
\[\det\left(s_{i+j}(F^{\boxplus 5})\right)_{i,j=0}^{17}<0,\quad\hbox{and}\quad
\det\left(s_{i+j}(F^{\boxplus 6})\right)_{i,j=0}^{41}<0.
\]

The characteristic polynomial for$F^{\boxplus t}(w)$ is
\[
\chi_{t}(w)=(1-w)^2\left(1-4w+6w^2-4t w^2-4w^3-2t w^3+w^4\right).
\]
The only solution of: $\chi_{t}(w)=\chi_{t}'(w)=0$, with $t>0$, is
\[
t_0=\frac{165\sqrt{33}-117}{128}\approx 6.49104,\qquad
w_0=\frac{\sqrt{33}-7}{2},
\]
and then
\[
\chi_{t_0}(w)=
\frac{(1-w)^2}{256}\left(\sqrt{33}-7 - 2 w\right)^2
\left(82 +14 \sqrt{33} - \left(587 + 101 \sqrt{33}\right) w + 64 w^2\right),
\]
so that all the roots of $\chi_{t_0}(w)$ are real. Therefore $F^{\boxplus t_0}\in\mathcal{F}_{\mathrm{rr}}^{0}$ and hence there exists a probability distribution for which the free cumulant sequence is
\[
r_n=\frac{165\sqrt{33}-117}{128}\cdot n^2.
\]
\end{example}

\section{Some examples of monotone convolution}\label{sec:monotone}

Now we will examine some examples of monotone convolutions. We will see that monotone convolution of two freely infinitely divisible can be no longer freely infinitely divisible. We particular, we will encounter some further elements of
$\mathcal{F}_{\mathrm{rr}}\setminus\mathcal{F}_{\mathrm{rr}}^{0}$ and some new examples of free deconvolution.

\begin{example}\label{ex:monotwmp}
Take $\mu:=\mathrm{W}(t)\rhd\mathrm{MP}(v,s)$, $s,t>0$, $v\in\mathbb{R}\setminus\{0\}$.
Then
\begin{align*}
F_{\mu}(w)&=F_{\mathrm{MP}(v,s)}\left(F_{\mathrm{W}(t)}(w)\right)
=\frac{w(1-v w+t w^2)}{(1+tw^2)(1-vw+svw+tw^2)},\\
\chi_{\mu}(w)&=\left(1-tw^2\right)\left(1-vw-\sqrt{s}vw+tw^2\right)
\left(1-vw+\sqrt{s}vw+tw^2\right),\\
R_{\mu}(w)&=s v w+tw^2+\frac{sv^2 w^2}{1-v w+tw^2}.
\end{align*}
Hence $\mu\in\mathcal{M}_{\mathrm{rr}}^{0}$ whenever  $4t\le\left(1-\sqrt{s}\right)^2 v^2$, otherwise $\mu\in\mathcal{M}_{\mathrm{rr}}\setminus\mathcal{M}_{\mathrm{rr}}^{0}$.
If $v^2>4t$ then, writing $1-v w+tw^2=(1-v_1 w)(1-v_2 w)$, with $v_1<v_2$,
$v_1+v_2=v$, $v_1 v_2=t$, we obtain
\[
R_{\mu}(w)=s(v_1+v_2)w+v_1 v_2 w^2+\frac{s(v_1+v_2)^2}{v_1(v_1-v_2)}\cdot\frac{v_1 w}{1-v_1 w}
+\frac{s(v_1+v_2)^2}{v_2(v_2-v_1)}\cdot\frac{v_2 w}{1-v_2 w}.
\]
Since $v_1 v_2=t>0$, one of the coefficients
\begin{equation}\label{for:c1c2}
c_1:=\frac{s(v_1+v_2)^2}{v_1(v_1-v_2)},\qquad
c_2:=\frac{s(v_1+v_2)^2}{v_2(v_2-v_1)}
\end{equation}
is positive, one is negative, therefore $\mu$ is not freely infinitely divisible but is
freely quasi infinitely divisible. In another words, for $v_1,v_2\in\mathbb{R}$, with
$v_1<v_2$, $v_1 v_2>0$, $s>0$, we can write symbolically:
\begin{equation}
\mathrm{W}(v_1 v_2)\rhd\mathrm{MP}(v_1+v_2,s)
=\delta_{s(v_1+v_2)}\boxplus\mathrm{W}(v_1 v_2)
\boxplus\mathrm{MP}(v_1,c_1)\boxplus\mathrm{MP}(v_2,c_2),
\end{equation}
where $c_1,c_2$ are given by (\ref{for:c1c2}).
\end{example}

\begin{example}\label{ex:monotmpw}
Now take $\mu:=\mathrm{MP}(v,s)\rhd\mathrm{W}(t)$, $s,t>0$, $v\in\mathbb{R}\setminus\{0\}$. Then
\begin{multline*}
F_{\mu}(w)=F_{\mathrm{W}(t)}\left(F_{\mathrm{MP}(v,s)}(w)\right)\\
=\frac{w (1 - v w) \big(1 - v(1-s) w\big)}
{1-2v(1-s)w+tw^2+v^2(1-s)^2w^2-2tvw^3+tv^2w^4}
\end{multline*}
and
\begin{multline*}
\chi_{\mu}(w)=\left(1-2vw+v^2(1-s)w^2\right)\left(1-\sqrt{t}w - v(1-s) w  + \sqrt{t} v w^2\right)\times\\
\times\left(1+\sqrt{t}w - v(1-s) w  - \sqrt{t} v w^2\right).
\end{multline*}
One can check that $\mu\in\mathcal{M}_{\mathrm{rr}}^{0}$
if an only if \[2|v|(1+s)\sqrt{t}\le t+v^2(1-s)^2,\]
otherwise $\mu\in\mathcal{F}_{\mathrm{rr}}\setminus\mathcal{F}_{\mathrm{rr}}^{0}$.

For $s=1$ we have
\[
R_{\mu}(w)=t w^2 - t v w^3 +\frac{vw}{1-vw},
\]
so if $27v^2\le t$ then, by Proposition~\ref{prop:rpolynomial3}, there is $\nu\in\mathcal{M}_{\mathrm{ratio}}$ such that $R_{\nu}(w)=t w^2 - t v w^3,$
and then $\mu=\nu\boxplus\mathrm{MP}(v,1)$.

Now assume that $0<s\ne1$.
Then the $R$-transform can be decomposed as
\[
R_{\mu}(w)=\frac{stw}{(1-s)^2 v}+\frac{tw^2}{1-s}+\frac{svw}{1-vw}
+\frac{st}{(s-1)^3v^2}\cdot\frac{v(1-s)w}{1-v(1-s)w},
\]
so accordingly we can write
\begin{equation}
\mu=\delta_{c_0}\boxplus\mathrm{W}(c_1)\boxplus\mathrm{MP}(v,s)\boxplus\mathrm{MP}(v(1-s),c_2),
\end{equation}
with
\begin{equation}
c_0=\frac{st}{(1-s)^2 v},\quad
c_1=\frac{t}{1-s},\quad
c_2=\frac{st}{(s-1)^3v^2}.
\end{equation}
Note that $c_1\cdot c_2<0$.
\end{example}

For example, the moment sequence of $\mathrm{MP}(1,1)\rhd\mathrm{W}(1)$
is $A242566$.

\begin{example}\label{ex:monotmpmp}
Now let us examine $\mu:=\mathrm{MP}(u,s)\rhd\mathrm{MP}(v,t)$.
Then
\[
F_{\mu}(w)=\frac{w(1-uw)\left(1-u(1-s)w-vw+uvw^2\right)}
{\big(1-u(1-s)w\big)\big(1-u(1-s)w-v(1-t)w+uv(1-t)w^2\big)},
\]
\begin{multline*}
\chi_{\mu}(w)=\left(1-u(1-\sqrt{s})w\right)\left(1-u(1+\sqrt{s})w\right)\times\\
\times\left(1-u(1-s)w-v(1-\sqrt{t})w+uv(1-\sqrt{t})w^2\right)\times\\
\times\left(1-u(1-s)w-v(1+\sqrt{t})w+uv(1+\sqrt{t})w^2\right),
\end{multline*}
\[
R_{\mu}(w)=\frac{suw}{1-uw}+\frac{tvw+tuv(s-1)w^2}{1-(u+ v-su)w+uv w^2}.
\]
If $(u-su+v)^2>4uv$ then we can write
\[
R_{\mu}(w)=\frac{suw}{1-uw}+\frac{a_{-} u_{-} w}{1-u_{-} w}+\frac{a_{+} u_{+} w}{1-u_{+} w},
\]
where
\[
u_{\mp}=\frac{1}{2}\left(u-su+v\mp\sqrt{(u-su+v)^2-4uv}\right)
\]
so that $u_{-}+u_{+}=u-su+v$, $u_{-} u_{+}=uv$ and $u_{-}<u_{+}$,
and
\[
a_{\mp}=t\frac{(1-s)\sqrt{(u-su+v)^2-4uv}\pm (1-s)^2 u\mp(1+s)v}{2\sqrt{(u-su+v)^2-4uv}}.
\]
Since
\[
(1-s)^2\left(u-su+v)^2-4uv\right)-\left((1-s)^2 u-(1+s)v\right)^2=-4sv^2<0,
\]
we have $a_{-}\cdot a_{+}<0$ and in this case
\[
\mu=\mathrm{MP}(u,s)\boxplus\mathrm{MP}(u_{-},a_{-})\boxplus\mathrm{MP}(u_{+},a_{+}),
\]
with one of $a_{-},a_{+}$ positive, one negative.
\end{example}

\begin{example}\label{ex:monotww}
In the case $\mu:=\mathrm{W}(s)\rhd\mathrm{W}(t)$ we have
\begin{align*}
F_{\mu}(w)&=\frac{w(1+sw^2)}{1+2sw^2+tw^2+s^2 w^4},\\
R_{\mu}(w)&=sw^2+\frac{tw^2}{1+sw^2},\\
\intertext{so that $\mu$ is a symmetric distribution, and}
\chi_{F}(w)&=\left(1-sw^2\right)\left(1-\sqrt{t}w+sw^2\right)\left(1+\sqrt{t}w+sw^2\right).
\end{align*}
Therefore $\mu\in\mathcal{M}_{\mathrm{rr}}^{0}$ if and only if $t\ge4s$.
For
\[
R_1(w)=\frac{tw^2}{1+sw^2}
\]
the corresponding element of $\mathcal{F}$ is
\[
F_1(w)=\frac{w(1+sw^2)}{1+sw^2+tw^2},
\]
and $F_1\in\mathcal{F}_{\mathrm{rr}}^{0}$ if and only if $t\ge8s$.
The distribution $\mu_1$ corresponding to $s=1,t=8$ is described in Example~\ref{ex:poly}.
In particular,
\[
\mathrm{W}(1)\rhd\mathrm{W}(8t)=\mathrm{W}(1)\boxplus\mu_1^{\boxplus t}
\]
for $t\ge1$.

For the distribution $\mu:=\mathrm{W}(1)\rhd\mathrm{W}(1)$, with $A007852$ as the even moment sequence $s_{2n}(\mu)$, the density function was computed in~\cite{crismale2020}.
\end{example}

Note that in these examples the resulting distribution $\mu$ is not freely infinitely divisible.
Let us also write for the record that for $t>0$, $u,v\ne0$, $u\ne v$ we have
\begin{align}
\delta_{u}\rhd\mathrm{W}(t)&=\delta_{(t+u^2)/u}\boxplus\mathrm{MP}\left(-u,t/u^2\right),\\
\delta_{u}\rhd\mathrm{MP}(v,t)&=\delta_{u(v-u-tv)/(v-u)}\boxplus\mathrm{MP}\left(v-u,tv^2/(v-u)^2\right),\\
\delta_{u}\rhd\mathrm{MP}(u,t)&=\delta_{(1+t)u}\boxplus\mathrm{W}\left(tu^2\right),
\end{align}
and that for an arbitrary distribution $\mu$ and for $u\in\mathbb{R}$ we have $\mu\rhd\delta_u=\mu\boxplus\delta_u$, the translation of $\mu$ by $u$.

\section{Polynomials in $\mathcal{F}$}\label{sec:polynomials}

In this section we will study polynomials as elements of $\mathcal{F}$, i.e. the case when $Q(w)=1$, so that $F(w)=wP(w)$,
\begin{align*}
R(w)&=\frac{1-P(w)}{P(w)},\\
F^{\boxplus t}(w)&=\frac{w P(w)}{P(w)+t(1-P(w))},\\
\intertext{and the characteristic polynomial of $F^{\boxplus t}(w)$ is}
\chi_t(w)&=P(w)^2+t\left[P(w)+w P'(w)-P(w)^2\right].
\end{align*}
This class was thoroughly studied in \cite{LM2020}. Here we will also study free powers of such functions.

\subsection{The case $\deg P(w)=1$}
If $P(w)=1-aw$, with $a\ne0$, then
\[
R(w)=\frac{aw}{1-aw},
\]
so that $F^{\boxplus t}\in\mathcal{F}_{\mathrm{rr}}^{0}$ for every $t>0$ and $\mu\left(F^{\boxplus t}\right)=\mathrm{MP}(a,t)$.

\subsection{The case $\deg P(w)=2$}
Take $P(w)=1+a w+b w^2$, $Q(w)=1$.
It is known \cite{LM2020} that in this case $F\in\mathcal{F}_{\mathrm{dist}}$ if and only if $a^2\ge3b$.
The characteristic polynomial for $F^{\boxplus t}$ is
\[
\chi_t(w)=1 + 2 a w + (a^2-a^2t+2b+bt)w^2 + 2 a b(1-t)w^3 + b^2(1-t)w^4.
\]

From now on we substitute:
\[
a:=a(u,v)=\frac{(2 + u)v}{1 + u + u^2},\qquad
b:=b(u,v)=\frac{v^2}{1 + u + u^2},\qquad
t:=u^3,
\]
which leads to
\[
\chi_t(w)=\frac{1}{1+u+u^2}(1+vw)^2\left(1+u+u^2+2v(1-u^2)w+v^2(1-u)w^2\right).
\]
If either $u\le-2$ or $u\ge1$ then all the roots of $\chi_{t}(w)$ are real.
Note that
\[
\frac{b(u,v)}{a(u,v)^2}=\frac{1+u+u^2}{(2+u)^2}
\]
increases with $u\ge1$ and
\[
\left\{b(u,v)/a(u,v)^2:u\ge1\right\}=[1/3,1),
\]
which means that if $a^2/3<b<a^2$ then $F\notin\mathcal{F}_{\mathrm{dist}}$
but $F^{\boxplus t}\in\mathcal{F}_{\mathrm{rr}}^0$ for certain $t>1$.

For the case $a^2=b$, say $a=b=1$, the moments $s_n\left(F^{\boxplus t}\right)$ of $F^{\boxplus t}$
are
\[
1, -t, t^2, t - t^3, -t - 4 t^2 + t^4,\ldots,
\]
and $\det(s_{i+j})_{i,j=0}^{2}=-t^2$, therefore $F^{\boxplus t}\notin\mathcal{F}_{\mathrm{dist}}$ for $t\ne0$.

Now let us now consider $u<-2$.
Then the function $(1+u+u^2)/(2+u)^2$ is increasing on $(-\infty,-2)$
and $\left\{b(u,v)/a(u,v)^2:u<-2\right\}=(1,+\infty)$.
It implies, that if $b>a^2>0$ then there exists $t<-8$ such that $F^{\boxplus t}\in\mathcal{F}_{\mathrm{rr}^{0}}$.

\begin{figure}[htp]
    \centering
    \includegraphics[width=10cm]{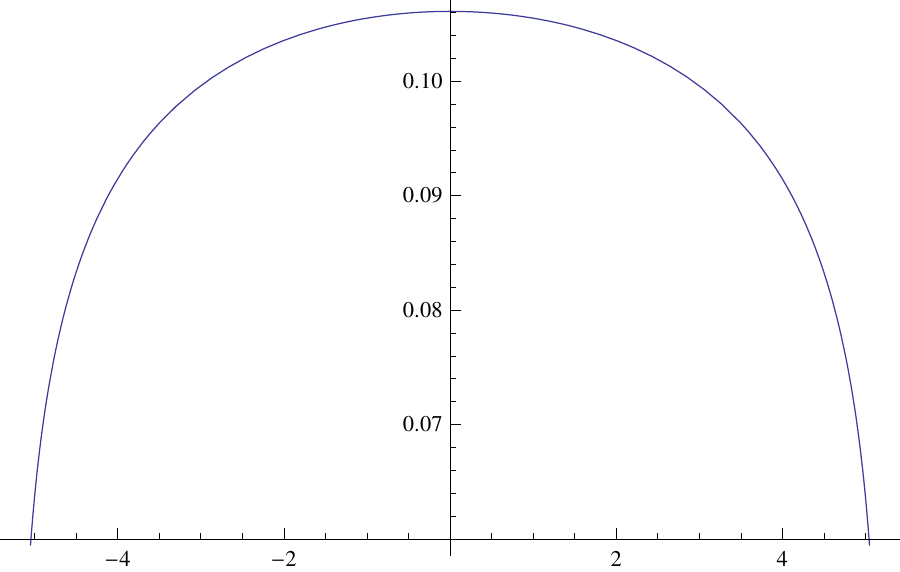}
    \caption{The density function of $\mu(F_1)$ for $F_1(w)=w(1+w^2)/(1+9w^2)$, see Example~\ref{ex:poly}}
    \label{fig:density}
\end{figure}

\begin{example}\label{ex:poly}
Let us examine in details the case $u=-2$, for convenience take $v=\sqrt{3}$,
so that $a=0$, $b=1$, $t=-8$. Then
\[
F(w)=w+w^3\qquad\hbox{and}\qquad
F_1(w):=F^{\boxplus -8}(w)=\frac{w(1+w^2)}{1+9w^2}.
\]
Then $\chi_{F_1}(w)=\left(1-3w^2\right)^2$, $R_1(w)=8w^2/\left(1+w^2\right)$, and the upshifted moment generating function $D_1(z)$ corresponding to $F_1$ is
\begin{equation}
D_1(z)=\frac{\sqrt{3}}{3}\cdot\frac{\sqrt[3]{1+z\sqrt{27}}-\sqrt[3]{1-z\sqrt{27}}}{\sqrt[3]{1+z\sqrt{27}}+\sqrt[3]{1-z\sqrt{27}}}.
\end{equation}
The distribution $\mu(F_1)$ is absolutely continuous with the density function
\begin{equation}
\frac{1}{\pi}\cdot\frac{1}{\sqrt[3]{\frac{\sqrt{27}+x}{\sqrt{27}-x}}+1+\sqrt[3]{\frac{\sqrt{27}-x}{\sqrt{27}+x}}}
\end{equation}
for $x^2\le27$, see Figure~\ref{fig:density}. This measure is symmetric, so the odd moments $s_{2n+1}(F_1)$ are all $0$,
the sequence of even moments $s_{2n}(F_1)$ is:
\[
1, 8, 120, 2184, 43768, 929544, 20524920, 466043784, 10807262712,
\ldots.
\]
\end{example}

\subsection{The case $\deg P(w)=3$}
Assume that $P(w)=1+aw+bw^2+cw^3$, $Q(w)=1$, so that $F(w)=wP(w)$.
By (\ref{for:necessary}) a necessary condition for $F$ to be in $\mathcal{F}_{\mathrm{dist}}$ is $b\le a^2$, see Figure~\ref{fig:deg3}.
Then
\[
\chi_{F}(w)=1+2aw+3bw^2+4cw^3.
\]
The following proposition is a reformulation of Theorem~7.2 in \cite{LM2020} in our notation.

\begin{proposition}
The function $F(w)=w(1+aw+bw^2+cw^3)$ belongs to $\mathcal{F}_{\mathrm{rr}}^{0}$ if and only if
\[
9a^2 b^2-27b^3-32a^3 c+108abc-108c^2\ge0.
\]
\end{proposition}

\begin{proof}
The roots of $\chi_{F}'(w)$ are
\[
w_{\mp}=\frac{-3b\mp\sqrt{3(3b^2-8ac)}}{12c}
\]
and $\chi_{F}(w)$ has only real roots if and only if $\chi_{F}(w_{-})\chi_{F}(w_{+})\le0$. Since
\[
\chi_{F}(w_{-})\cdot\chi_{F}(w_{+})=\frac{-9a^2 b^2+27b^3+32a^3 c-108abc+108c^2}{108c^2},
\]
the statement holds.
\end{proof}

\begin{figure}[htp]
    \centering
    \includegraphics[width=10cm]{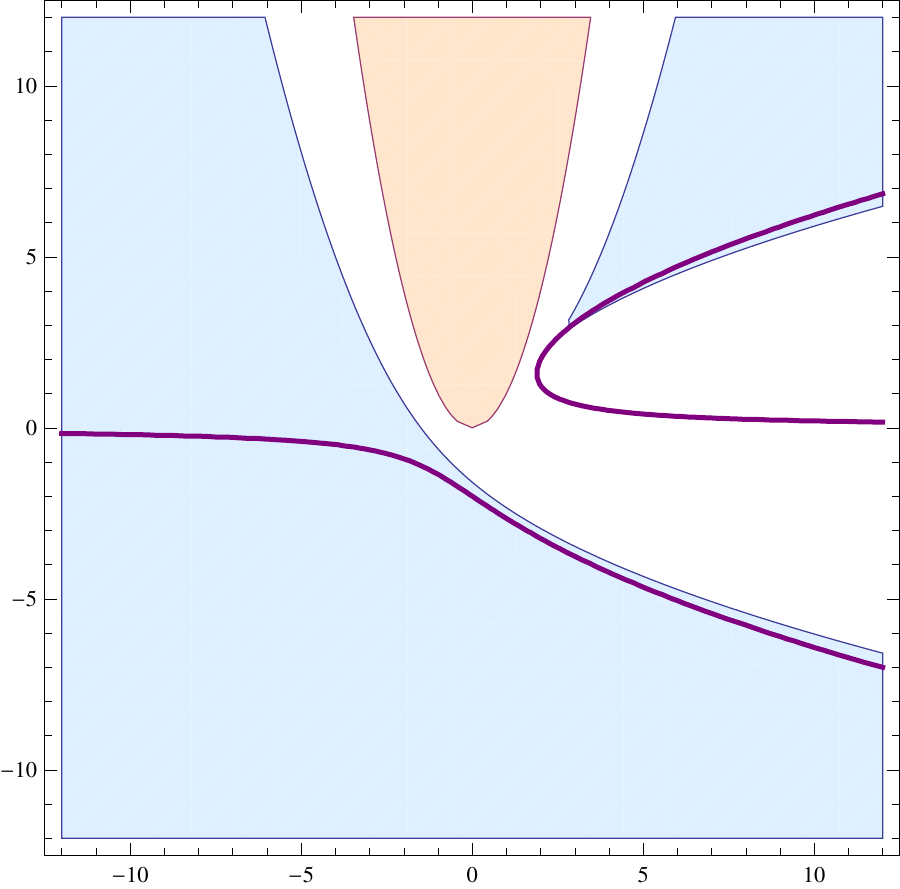}
    \caption{For $F(w)=w(1+aw+bw^2+w^3)$ the blue area corresponds to these $(a,b)$ that $F\in\mathcal{F}_{\mathrm{rr}}^{0}$. If $b>a^2$, the points above the parabola, then $F\notin\mathcal{F}_{\mathrm{dist}}$. The part of the thick purple line which is outside the blue area corresponds to $F\in\mathcal{F}_{\mathrm{rr}}\setminus\mathcal{F}_{\mathrm{rr}}^{0}$.}
    \label{fig:deg3}
\end{figure}

Now let us examine monotone convolutions within the class of polynomials of order~3.

\begin{example}
Take $F_1(w):=w(1+uw)$, $F_2(w):=w(1+vw)$ and $F(w):=F_2\left(F_1(w)\right)$, $u,v\in\mathbb{R}$, $u,v\ne0$. Then
\begin{align*}
F(w)&=w(1+uw)\left(1+vw+uvw^2\right),\\
\chi_{F}(w)&=(1+2uw)\left(1+2vw+2uvw^2\right),\\
R(w)&=\frac{-uw}{1+uw}-\frac{vw}{1+vw+uvw^2}.
\end{align*}
(This is in fact a special case of Example~\ref{ex:monotmpmp}, with $s=t=1$.)
In this way we obtain all $F(w)=w(1+aw+bw^2+cw^3)$ with $a,b,c\in\mathbb{R}$, $c\ne0$ and $b^3+8c^2=4abc$, see the thick purple line on Figure~\ref{fig:deg3}.
If $v^2<2uv$ (i.e. if $b>0$ and $b^3<16b^2$) then $F\in\mathcal{F}_{\mathrm{rr}}\setminus\mathcal{F}_{\mathrm{rr}}^{0}$ and these are all elements of $\mathcal{F}_{\mathrm{rr}}\setminus\mathcal{F}_{\mathrm{rr}}^{0}$ within $\mathcal{F}(3,0)$,
see Proposition~7.3 in~\cite{LM2020}.

\end{example}

\section{Further examples}

In this part we provide three families of integer sequences that appear in OEIS and are of the form $s_n(F)$ for some $F\in\mathcal{F}_{\mathrm{rr}}^{0}$. In particular we encounter all sequences from Figure~11 in \cite{bostan2020},
A selection of examples where $F(w)$ is a polynomial is provided in~\cite{LM2020}.

\begin{example}
Let us examine moment sequences of distributions of the form
\[\mu=\mathrm{MP}(u,s)\boxplus\delta_{v},\]
the Marchenko-Pastur law $\mathrm{MP}(u,s)$, $s>0,u\ne0$ (see Section~\ref{sec:inf}), translated by $v\in\mathbb{R}$. Then
\begin{align*}
R_{\mu}(w)&=\frac{suw}{1-uw}+vw,\\
F_{\mu}(w)&=\frac{w(1-uw)}{1-uw+suw+vw-uvw^2},\\
\chi_{F}(w)&=1 - 2 u w + u^2 w^2 - s u^2 w^2,
\end{align*}
and $F\in\mathcal{F}_{\mathrm{rr}}^{0}$. Several examples of such sequences
can be found in OEIS:
$A000108$ for $(s,u,v)=(1,1,0)$ (Catalan numbers),
$A007317$ for $(1,1,1)$,
$A064613$ for $(1,1,2)$,
$A104455$ for $(1,1,3)$,
$A104498$ for $(1,1,4)$,
$A154623$ for $(1,1,5)$,
$A005043$ for $(1,1,-1)$ (Riordan numbers),
$A126930$ for $(1,1,-2)$,
$A168491$ for $(1,-1,0)$, (signed Catalan numbers),
$A099323$ for $(1,-1,1)$,
$A001405$ for $(1,-1,2)$,
$A005773$ for $(1,-1,3)$,
$A001700$ for $(1,-1,4)$,
$A026378$ for $(1,-1,5)$,
$A005573$ for $(1,-1,6)$,
$A122898$ for $(1,-1,7)$,
$A151374$ for $(1,2,0)$,
$A162326$ for $(1,2,1)$,
$A337168$ for $(1,2,-1)$,
$A060899$ for $(1,-2,4)$,
$A151318$ for $(1,-2,5)$,
$A005159$ for $(1,3,0)$,
$A337167$ for $(1,3,1)$,
$A337169$ for $(1,3,-1)$,
$A151403$ for $(1,4,0)$,
$A156058$ for $(1,5,0)$,
$A006318$ for $(2,1,0)$ (large Schr\"{o}der numbers),
$A174347$ for $(2,1,1)$,
$A052709$ for $(2,1,-1)$,
$A126087$ for $(2,-1,3)$,
$A151282$ for $(2,-1,4)$,
$A151090$ for $(2,-1,5)$,
$A225887$ for $(2,-1,6)$,
$A156017$ for $(2,2,0)$,
$A047891$ for $(3,1,0)$,
$A064641$ for $(3,1,-1)$,
$A129147$ for $(3,1,-2)$,
$A128386$ for $(3,-1,4)$,
$A151292$ for $(3,-1,5)$,
$A082298$ for $(4,1,0)$,
$A062992$ for $(4,1,-1)$,
$A129148$ for $(4,1,-2)$,
$A330800$ for $(4,-1,3)$,
$A121724$ for $(4,-1,5)$,
$A344558$ for $(4,-1,6)$,
$A330799$ for $(4,-1,7)$,
$A194723$ for $(4,-1,9)$,
$A082301$ for $(5,1,0)$,
$A128387$ for $(5,-1,6)$,
$A118376$ for $(1/2,2,1)$,
$A103210$ for $(3/2,2,0)$,
$A306519$ for $(1/2,2,-1)$,
$A151281$ for $(1/2,-2,3)$,
$A129637$ for $(1/2,-2,4)$,
$A151251$ for $(1/2,-2,5)$,
$A134425$ for $(1/2,-2,6)$,
$A103211$ for $(4/3,3,0)$,
$A000957$ for $(1/4,2,-1/2)$ (Fine numbers),
$A133305$ for $(5/4,4,0)$,
$A033321$ for $(1/4,2,1/2)$,
$A344507$ for $(1/4,2,-3/2)$,
$A033543$ for $(1/4,2,3/2)$,
$A054341$ for $(1/4,-2,5/2)$,
$A059738$ for $(1/4,-2,7/2)$,
$A049027$ for $(1/4,-2,9/2)$,
$A133306$ for $(6/5,5,0)$,
$A133307$ for $(7/6,6,0)$,
$A133308$ for $(8/7,7,0)$,
$A133309$ for $(9/8,8,0)$.

Putting $(s,u,v)=(1/k,k,0)$, $k=2,3,\ldots,11$, we get
$A001003$, $A007564$, $A059231$, $A078009$, $A078018$, $A081178$, $A082147$, $A082181$,
$A082148$, $A082173$, respectively.
\end{example}

\begin{example}
Now consider distributions of the form
\[
\mu=\mathrm{W}(s)\boxplus\delta_{u},
\]
the Wigner law with parameter $s>0$ (see Section~\ref{sec:inf}),
translated by $u\in\mathbb{R}$. Then
\[
R_{\mu}(w)=uw+sw^2,\qquad
F_{\mu}(w)=\frac{w}{1+uw+s w^2}.
\]
For $u=0$ we get the aerated sequence $\binom{2n+1}{n}s^n/(2n+1)$, in particular for $s=1,2,3,4,5,6,7,8,9,10$ we get the aerated sequence
$A000108$ (Catalan numbers, see also $A126120$),
$A151374$, $A005159$, $A151403$, $A156058$, $A156128$, $A156266$, $A156270$, $A156273$, $A156275$, respectively.

We also get:
$A001006$ for $(s,u)=(1,1)$,
$A000108$ for $(1,2)$,
$A002212$ for $(1,3)$,
$A005572$ for $(1,4)$,
$A182401$ for $(1,5)$,
$A025230$ for $(1,6)$,
$A025235$ for $(2,1)$,
$A071356$ for $(2,2)$,
$A001003$ (little Schr\"{o}der numbers) for $(2,3)$,
$A068764$ for $(2,4)$,
$A025237$ for $(3,1)$,
$A122871$ for $(3,2)$,
$A107264$ for $(3,3)$,
$A007564$ for $(3,4)$,
$A068765$ for $(3,6)$, 
$A091147$ for $(4,1)$,
$A129400$ for $(4,2)$,
$A003645$ for $(4,4)$,
$A059231$ for $(4,5)$,
$A068766$ for $(4,8)$,
$A091148$ for $(5,1)$,
$A249925$ for $(5,2)$,
$A107265$ for $(5,5)$,
$A078009$, $A127848$ for $(5,6)$,
$A068767$ for $(5,10)$,
$A091149$ for $(6,1)$,
$A269730$ for $(6,5)$,
$A107266$ for $(6,6)$,
$A078018$ for $(6,7)$,
$A068768$ for $(6,12)$,
$A217275$ for $(7,1)$,
$A081178$ for $(7,8)$,
$A068769$ for $(7,14)$,
$A090442$ for $(8,6)$,
$A082147$ for $(8,9)$,
$A068770$ for $(8,16)$,
$A132900$ for $(9,3)$,
$A101600$ for $(9,6)$,
$A082181$ for $(9,10)$,
$A068771$ for $(9,18)$,
$A068772$ for $(10,20)$.

For $s=\gamma(1+\gamma)$, $u=2\gamma+1$, the sequence $s_n(F)$ counts
dimensions of the $\gamma$-polytridendriform operad $\mathrm{TDendr}_\gamma$, see~\cite{giraudo2016}, which leads to $A001003$, $A269730$, $A269731$, $A269732$ for $\gamma=1,2,3,4$ respectively
\end{example}

\begin{example}
Take $F(w):=w(1+w)(1+2w+2w^2)$, then $\chi_{F}(w)=(1+2w)^3$.
In fact, the resulting distribution $\mu$ is the reflection of that corresponding to Example~\ref{ex:A048779}, see entry $A048779$ in OEIS.
Then the moment sequences of $\mu\boxplus\delta_{4}$ and $\mu\boxplus\delta_{6}$ are $A305608$ and $A151323$, respectively.
\end{example}


\end{document}